\documentclass[11pt]{article}
\usepackage{amsmath,amssymb,amsthm,geometry,hyperref}
\numberwithin{equation}{section}
\allowdisplaybreaks[2]
\hypersetup{colorlinks=true,linkcolor=blue,citecolor=blue,urlcolor=blue,pdftitle={Green-Function Monotonicity under Positive Bakry-Emery Curvature},pdfsubject={Weighted Green functions, monotonicity, rigidity, and negative effective dimension}}

\usepackage{txfonts}

\newcommand{\Ric}{\operatorname{Ric}}
\newcommand{\Hess}{\operatorname{Hess}}

\newtheorem{theorem}{Theorem}[section]
\newtheorem{proposition}[theorem]{Proposition}
\newtheorem{lemma}[theorem]{Lemma}
\newtheorem{corollary}[theorem]{Corollary}
\theoremstyle{remark}
\newtheorem{remark}[theorem]{Remark}

\begin{document}

\title{Green-Function Monotonicity under Positive Bakry--\'Emery Curvature 
}
\author{Yu-Zhao Wang\footnote{Y.-Z. Wang was supported by the Fundamental Research Program of Shanxi Province (Grant No. 202303021211001).} \quad Zeng-Ting Wu}
\date{}
\maketitle

\begin{abstract}
Inspired by the Green-function monotonicity formulas of Colding, Minicozzi, and Manea, we establish weighted monotonicity formulas on closed weighted Riemannian manifolds satisfying ${\rm Ric}_f^m \ge (m-1)kg$, where $m>n\ge 3$ and $k>0$. Our results extend Manea's elliptic monotonicity framework to the setting of finite-dimensional Bakry-\'Emery curvature and provide a positive-curvature counterpart to the weighted formulas of Song-Wei-Wu. Starting from the Green function of $-\Delta_f+m(m-2)k/4$, we derive an exact weighted Bochner identity and prove three monotonicity formulas, together with a one-parameter family for $\beta\ge (m-2)/(m-1)$, under explicit pole-integrability conditions. We also treat the case of negative effective dimension $m<0$, for which the nonnegative square decomposition remains valid, while the change in the behavior near the Green pole leads to reversed area and volume monotonicity on finite sublevel sets. We conclude with comparisons and  questions concerning entropy monotonicity formulas and positive-curvature RCD spaces.

\end{abstract}
\vspace{3mm}

\vspace{3mm}

\section{Introduction}\label{sec:intro}

\qquad Monotonicity formulae built from Green functions have become an important bridge between potential theory, comparison geometry, and the quantitative study of metric cones.  On a complete nonparabolic manifold with nonnegative Ricci curvature, Colding \cite{Colding} introduced three scale-invariant quantities associated with the positive Green function and proved that two of them are nonincreasing while a distinguished linear combination is nondecreasing .  These identities are tied to the sharp estimate \( |\nabla b|\leq 1 \) for \(b=G^{1/(2-n)}\), and their defect terms measure the failure of the level sets of \(b\) to be conical.  Colding and Minicozzi \cite{ColdingMinicozzi} subsequently obtained a one-parameter family and used its level-set geometry to derive further asymptotic information.  Among other applications, these formulae enter the analysis of uniqueness of tangent cones for Einstein manifolds \cite{unique}. Also refer to survey\cite{CM}.

The corresponding problem under a positive Ricci lower bound is subtler.  A closed manifold has no positive Green function for the Laplacian itself, and the Euclidean comparison function must be replaced by its spherical analogue. Recently,  Manea  \cite{Manea}  resolved both issues by using the Green function of
\[
 -\Delta+\frac{n(n-2)}4\kappa
\]
on a closed manifold satisfying \(\Ric\geq(n-1)\kappa g\).  He obtained a sharp spherical gradient estimate, three curvature-corrected monotonicity formulae, their one-parameter extensions, and a rigidity theorem characterizing the round sphere.  A decisive feature of his calculation is that, after writing
\(
G=\left(2\operatorname{sn}_{\kappa}(s/2)\right)^{2-n},
\)
the square of \(2\operatorname{sn}_{\kappa}(s/2)\) has a trace-free Hessian defect with a favorable sign.

Weighted analogues in nonnegative curvature were studied earlier by Song, Wei, and Wu \cite{SongWeiWu}.  On a complete \(f\)-nonparabolic smooth metric measure space with nonnegative finite-dimensional Bakry--\'Emery Ricci curvature, they introduced several parameters in the powers of the Green function and in the level and sublevel integrals.  Their framework contains the Colding and Colding--Minicozzi quantities and yields many additional families.
Gigli and Violo \cite{GigliViolo} subsequently obtained a related
parameter threshold in the $\mathrm{RCD}(0,m)$ harmonic setting. 

 The purpose of the present paper is to determine precisely which parts of Manea's positive-curvature theory survive under the finite-dimensional Bakry--\'Emery condition \cite{BakryEmery}
\begin{equation}\label{BE}
 \Ric_f^m:=\Ric+\Hess f-\frac{df\otimes df}{m-n}\geq(m-1)kg,\quad m>n\geq3.
\end{equation}
and to identify the new obstructions caused by the weighted structure.  Our main observation is that the correct replacement of the ordinary trace-free Hessian is not merely
\(
\Hess h-(\Delta_fh/m)g.
\)
Instead one must retain the complete finite-dimensional remainder in the generalized Bochner formula :
\begin{equation}\label{D}
 \mathcal D(h):=|\mathring{\Hess}\ h|^2+
 \frac{n}{m(m-n)}\left(\frac{m-n}{n}\Delta h+
                 \langle\nabla f,\nabla h\rangle\right)^2
 +\bigl(\Ric_f^m-(m-1)kg\bigr)(\nabla h,\nabla h),
\end{equation}
where  $\mathring{\Hess}\ h=\Hess h-(\Delta h/n)g$ is the trace-free tensor. This is exactly the nonnegative defect singled out by the \(\Gamma_2\) formalism; see Villani \cite[Chapter~14]{Villani}.  It remains meaningful both for \(m>n\) and for \(m<0\).

Let \(G\) be the Green function with pole \(p\) of
\begin{equation}\label{Deltaf}
 -\Delta_f+\frac{m(m-2)}4k,
\end{equation}
 define
\begin{equation}\label{sv}
s=\tfrac12G^{1/(2-m)},\qquad
v=(4|\nabla s|^2+ks^2)^{1/2},\qquad
\alpha=\frac{n-2}{m-2}.
\end{equation}
For a smooth weight, $G\sim c_p\rho^{2-n}$ at its pole, where $\rho=d_g(p,\cdot)$. Consequently $s\asymp\rho^\alpha$ and $v\asymp\rho^{\alpha-1}$. If $m>n$, then $\alpha<1$ and $v$ diverges. Replacing the unweighted dimension by $m$ in the differential calculation does not change the actual-dimensional fundamental solution. In particular, the unweighted finite pole normalization cannot be transferred to this problem.

\paragraph{Main results.}
Throughout Sections~\ref{sec:setup}--\ref{sec:rigidity}, the manifold is closed and connected, $f$ is smooth, $m>n\ge3$, and $\operatorname{Ric}_f^m\ge(m-1)kg$ with $k>0$. The source mass of $G$ is prescribed by \eqref{eq:2.5}. Section~\ref{sec:setup} proves the exact identity
\begin{equation}\label{Ls}
\mathcal L_f(4|\nabla s|^2+ks^2)
=\frac{2}{s^2}\mathcal{D}(s^2),
\end{equation}
where $\mathcal L_f=G^{-2}\operatorname{div}_f(G^2\nabla\cdot)$. No pole integrability is needed for this identity on the punctured manifold.

Inspired by Manea's work in \cite{Manea}, the weighted level-set quantities $A_f,V_f$ are defined as follows:
\begin{equation}\label{eq:4.5}
\begin{aligned}
A_f(r)
:={}&
(2\operatorname{sn}_k(r/2))^{1-m}\operatorname{cs}_k(r/2)
\int_{b_f=r}(4|\nabla s|^2+ks^2)|\nabla b_f|e^{-f}d\sigma\\
&+\frac{mk}{4}\int_{b_f<r}(4|\nabla s|^2+ks^2)G e^{-f}dV,\\
V_f(r)
:={}&
(2\operatorname{sn}_k(r/2))^{-m}
\int_{b_f<r}(4|\nabla s|^2+ks^2)(4|\nabla s|^2-ks^2)e^{-f}dV\\
&+\frac{k}{4}\int_{b_f<r}(4|\nabla s|^2+ks^2)G e^{-f}dV.
\end{aligned}
\end{equation}

\begin{theorem}[Theorem \ref{Theorem 3.6.} and \ref{Theorem 3.11.}]\label{3mono}
Let $(M,g,e^{-f}dV)$ be a closed weighted Riemannian manifold of dimension $n\ge3$ with $Ric^m_f\ge(m-1)kg$ for  $m\ge n$ and $k>0$. Let $G$ be weighted Green's function for the operator $-\Delta_f+m(m-2)k/4$, put $s=\tfrac12G^{1/(2-m)}$,  then 
\begin{align}\label{DA}
A_f'(r)&=
-8(2\operatorname{sn}_k(r/2))^{m-3}\operatorname{cs}_k(r/2)\\&\cdot\int_{s>\operatorname{sn}_k(r/2)}
(2s)^{2-2m}\left[
|\mathring{\operatorname{Hess}}s^2|^2
+D_m
+(\operatorname{Ric}_f^m-(m-1)kg)(\nabla s^2,\nabla s^2)
\right]e^{-f}dV\le0,\notag\\
\label{DV}
V_f'(r)&=\frac{1}{2\tan_k(r/2)}(A_f(r)-mV_f(r))\le0,
\end{align}
and
\begin{equation}\label{eq:4.22}
\begin{aligned}
&A_f'(r)-2(m-1)V_f'(r)\\
&=
\frac{8\operatorname{cs}_k(r/2)}{(2\operatorname{sn}_k(r/2))^{m+1}}
\int_{b_f<r}\left[
|\mathring{\operatorname{Hess}}s^2|^2
+D_m
+(\operatorname{Ric}_f^m-(m-1)kg)(\nabla s^2,\nabla s^2)
\right]e^{-f}dV\ge0,
\end{aligned}
\end{equation}
where $D_m$ is defined by
\begin{equation}\label{Dm}
D_m:=
\frac{n}{m(m-n)}
\left(\frac{m-n}{n}\Delta s^2+\langle\nabla f,\nabla s^2\rangle\right)^2\ge0.
\end{equation}
Therefore \(A_f\) and \(V_f\) are nonincreasing and \(A_f-2(m-1)V_f\) is nondecreasing.
\end{theorem}
\begin{remark}
The first monotonicty in Theorem \ref{3mono} holds for every $m>n\ge3$. The statements involving $V_f$ require $n+4\alpha-4>0$, which is automatic for $n\ge4$ and is equivalent to $3<m<6$ for $n=3$, where $\alpha=\frac{n-2}{m-2}$. We retain both the divergence proof and the complete coefficient-cancellation proof.
\end{remark}
For the family in Section~\ref{sec:beta}, $A_{f,\beta}=I_{f,v^\beta}$ is nonincreasing when
\[
\frac{m-2}{m-1}\le\beta<\frac{2}{1-\alpha}.
\]
With $V_{f,\beta}=W_{f,v^\beta}$ and the stronger restriction
\[
\frac{m-2}{m-1}\le\beta<\frac{\min\{2,n-2\}}{1-\alpha},
\]
we also have
\begin{equation}\label{betaAV}
A_{f,\beta}'\le0,\qquad V_{f,\beta}'\le0,\qquad
\bigl(A_{f,\beta}-2(m-2)V_{f,\beta}\bigr)'\ge0.
\end{equation}
The two upper bounds are strict because their endpoints produce logarithmically divergent integrals. Although $A_{f,2}=A_f$, the volume kernels defining $V_{f,2}$ and $V_f$ differ. This distinction accounts for their different convergence conditions.

Section~\ref{sec:rigidity} extracts the complete equality conditions, including the lower endpoint of the $\beta$ range. On a regular equality region they force a sine-warped metric with a specified radial weight. These local conditions are incompatible with equality on the relevant full sublevel or exterior regions of the smooth closed problem when $m>n$. 

In Section~\ref{sec:negative}, we also examine \(m<0\), a range in which the generalized Bochner formula and the curvature-dimension condition remain valid; see Ohta \cite{Ohta}.   The exact squares retain their signs, but finite sublevel sets avoid the pole, and the three natural quantities are nondecreasing. Here $(m-1)k<0$, so this is a negative curvature lower bound, not a positive-curvature theorem in disguise.

These results give a direct weighted version of the positive-curvature calculation together with its analytic domain of validity. They also identify the extra information needed before one can discuss singular equality models, renormalized pole quantities, or extensions to metric measure spaces. Section~\ref{sec:comparison} compares the formulas with existing work and formulates these further questions. 

\section{The Green function and the weighted Bochner formulas}\label{sec:setup}

Let \((M^n,g, e^{-f}dV)\) be a closed, connected smooth weighed Riemannian manifold, where \(f\in C^\infty(M)\), let \(m>n\ge 3\), and  \(k>0\). The volume element and hypersurface area element are denoted by \(dV\) and \(d\sigma\), respectively. Define
\begin{equation}\label{eq:2.1}
\Delta_f u=\Delta u-\langle\nabla f,\nabla u\rangle,\qquad
\operatorname{Ric}_f^m=\operatorname{Ric}+\operatorname{Hess}f-\frac{df\otimes df}{m-n},
\end{equation}
where \(\Delta=\operatorname{div}\nabla\), and assume that
\begin{equation}\label{eq:2.2}
\operatorname{Ric}_f^m\ge (m-1)kg.
\end{equation}
For a vector field \(Z\), \(\operatorname{div}_f Z=e^f\operatorname{div}(e^{-f}Z)\), hence
\begin{equation}\label{eq:2.3}
\int_U \operatorname{div}_f Z\,e^{-f}dV
=
\int_{\partial U}\langle Z,\nu\rangle e^{-f}d\sigma.
\end{equation}
Fix \(p\in M\). Let \(G\) be the positive Green function of
\begin{equation}\label{eq:2.4}
L_f=-\Delta_f+\frac{m(m-2)k}{4},
\end{equation}
with the explicit weak normalization
\begin{equation}\label{eq:2.5}
\int_M G L_f\phi\,e^{-f}dV=\kappa\phi(p),\qquad \kappa>0,\quad \phi\in C^\infty(M).
\end{equation}
The positivity of the operator and standard elliptic theory give existence and uniqueness of this Green function, and the maximum principle gives its positivity. On \(M\setminus\{p\}\),
\begin{equation}\label{eq:2.6}
\Delta_f G=\frac{m(m-2)k}{4}G.
\end{equation}

\subsection{The level-set parameter \(s\) and Manea's parameter}

Throughout we use only the abbreviation
\begin{equation}\label{eq:2.7}
s=\frac12 G^{1/(2-m)},\qquad G=(2s)^{2-m}.
\end{equation}
We extend continuously by \(s(p)=0\). Define
\[
\operatorname{sn}_k(r)=\frac{\sin(\sqrt{k}r)}{\sqrt{k}},\qquad
\operatorname{cs}_k(r)=\cos(\sqrt{k}r),\qquad
\tan_k=\frac{\operatorname{sn}_k}{\operatorname{cs}_k},\qquad
\operatorname{ct}_k=\frac{\operatorname{cs}_k}{\operatorname{sn}_k}.
\]
On the region \(s<1/\sqrt{k}\), define
\begin{equation}\label{eq:2.8}
b_f=2\arcsin_k(s),\qquad \arcsin_k(s)=\frac1{\sqrt k}\arcsin(\sqrt{k}s),\qquad s=\operatorname{sn}_k(b_f/2).
\end{equation}
On the level set \(b_f=r\), the value of \(s\) is \(\operatorname{sn}_k(r/2)\), and
\begin{equation}\label{eq:2.9}
2\nabla s=\operatorname{cs}_k(b_f/2)\nabla b_f,\qquad
\nu=\frac{\nabla s}{|\nabla s|}=\frac{\nabla b_f}{|\nabla b_f|}.
\end{equation}
In Section 3 we take \(r\) as the parameter and require
\begin{equation}\label{eq:2.10}
0<r<\frac{\pi}{\sqrt{k}},\qquad
\operatorname{sn}_k(r/2)<\max_M s.
\end{equation}
Then \(\operatorname{cs}_k(r/2)>0\). The notation \(\{b_f<r\}\) and \(\{b_f=r\}\) denotes respectively \(\{s<\operatorname{sn}_k(r/2)\}\) and \(\{s=\operatorname{sn}_k(r/2)\}\); all these sets lie in the region where \(b_f\) is defined. The exterior region is always written as \(\{s>\operatorname{sn}_k(r/2)\}\), so that one need not assume without proof that \(b_f\) is defined on the whole manifold. If a global range estimate \(s\le 1/\sqrt{k}\) is available, then the exterior region can also be written as \(\{b_f>r\}\) as in the original paper. This paper does not assume this unproved range estimate.

\subsection{Pole asymptotics and normalization constant}

Let \(\rho=d_g(p,\cdot)\) and \(\omega_{n-1}=|\mathbb S^{n-1}|\). We use the standard local asymptotics for the Green function of a smooth elliptic operator:
\begin{equation}\label{eq:2.11}
\begin{aligned}
G&=c_p\rho^{2-n}+o(\rho^{2-n}),\qquad
c_p=\frac{\kappa e^{f(p)}}{(n-2)\omega_{n-1}},\\
\nabla^j(G-c_p\rho^{2-n})&=o(\rho^{2-n-j}),\qquad j=1,2,3.
\end{aligned}
\end{equation}
The derivative remainders are uniform on annuli in rescaled normal coordinates. The polar parametrix applies because
\begin{equation}\label{eq:2.12}
e^{-f/2}L_f(e^{f/2}v)
=
-\Delta v+
\left(\frac{m(m-2)k}{4}+\frac{|\nabla f|^2}{4}-\frac{\Delta f}{2}\right)v.
\end{equation}
The right-hand side is a smooth Schrödinger operator whose principal part is the same as the Laplacian. The above asymptotics follow from a local fundamental-solution construction and elliptic interior estimates on rescaled annuli; we treat this as standard analytic input and do not repeat the parametrix construction. From the weak normalization one can also independently verify that
\begin{equation}\label{eq:2.13}
-\lim_{\epsilon\downarrow0}\int_{\rho=\epsilon}\langle\nabla G,\nabla\rho\rangle e^{-f}d\sigma
=
(n-2)c_p e^{-f(p)}\omega_{n-1}
=
\kappa.
\end{equation}
For stating the asymptotic orders, we retain the necessary exponent
\begin{equation}\label{eq:2.14}
\alpha=\frac{n-2}{m-2}\in(0,1).
\end{equation}
By the chain rule,
\begin{equation}\label{eq:2.15}
2s=c_p^{-1/(m-2)}\rho^\alpha(1+o(1)).
\end{equation}
\begin{equation}\label{eq:2.16}
4|\nabla s|^2+ks^2=\alpha^2c_p^{-2/(m-2)}\rho^{2\alpha-2}(1+o(1)).
\end{equation}
\begin{equation}\label{eq:2.17}
|\operatorname{Hess}s^2|=O(\rho^{2\alpha-2}),\qquad
|\nabla(4|\nabla s|^2+ks^2)|=O(\rho^{2\alpha-3}).
\end{equation}
Since the leading radial derivative is strictly positive, small level sets near \(p\) are radial graphs. Their geodesic radius satisfies
\[
\rho\sim \left(2c_p^{1/(m-2)}\operatorname{sn}_k(r/2)\right)^{1/\alpha},
\]
and the area element satisfies \(e^{-f}d\sigma=e^{-f(p)}\rho^{n-1}(1+o(1))\,d\omega\). Therefore the following level-set asymptotics are justified by the implicit function theorem and the derivative asymptotics.

\subsection{Weighted Bochner calculations}

\begin{lemma}On \(M\setminus\{p\}\),
\begin{equation}\label{eq:3.1}
\Delta_f s=(m-1)\frac{|\nabla s|^2}{s}-\frac{mk}{4}s.
\end{equation}
\begin{equation}\label{eq:3.2}
\Delta_f s^2=2m\left(|\nabla s|^2-\frac{k}{4}s^2\right).
\end{equation}
\begin{equation}\label{eq:3.3}
(\Delta_f+mk)s^2=2m\left(|\nabla s|^2+\frac{k}{4}s^2\right).
\end{equation}
\end{lemma}
\begin{proof}
The chain rule and \eqref{eq:2.7} give,
\[
\Delta_f G=2^{2-m}(2-m)\left(s^{1-m}\Delta_f s+(1-m)s^{-m}|\nabla s|^2\right).
\]
Dividing by \(G\) and using \eqref{eq:2.6}, we obtain
\[
\frac{m(m-2)k}{4}
=
(2-m)\left(\frac{\Delta_f s}{s}+(1-m)\frac{|\nabla s|^2}{s^2}\right),
\]
which is \eqref{eq:3.1}. Substituting this into \(\Delta_f(s^2)=2s\Delta_f s+2|\nabla s|^2\) gives the second identity; adding \(mks^2\) gives the third.
\end{proof}

%\subsection*{The coefficient of the weighted square term}

Define
\begin{equation}\label{eq:3.4}
D_m:=
\frac{n}{m(m-n)}
\left(\frac{m-n}{n}\Delta s^2+\langle\nabla f,\nabla s^2\rangle\right)^2\ge0.
\end{equation}
Let \(\mathring{\operatorname{Hess}}v=\operatorname{Hess}v-(\Delta v/n)g\), whose trace is taken with respect to the actual dimension \(n\). A direct expansion of the square verifies that
\begin{equation}\label{eq:3.5}
\begin{aligned}
&\frac{(\Delta s^2)^2}{n}
+\frac{\langle\nabla f,\nabla s^2\rangle^2}{m-n}
-\frac{(\Delta_f s^2)^2}{m}\\
&=
\frac{m-n}{mn}(\Delta s^2)^2
+\frac{2}{m}\Delta s^2\langle\nabla f,\nabla s^2\rangle
+\frac{n}{m(m-n)}\langle\nabla f,\nabla s^2\rangle^2
=D_m.
\end{aligned}
\end{equation}
Thus the exact weighted Bochner formula (see Villani \cite[Chapter~14]{Villani}) is
\begin{equation}\label{eq:3.6}
\begin{aligned}
\frac12\Delta_f|\nabla s^2|^2
=
\langle\nabla s^2,\nabla\Delta_f s^2\rangle
+\frac{(\Delta_f s^2)^2}{m}
+|\mathring{\operatorname{Hess}}s^2|^2+D_m+\operatorname{Ric}_f^m(\nabla s^2,\nabla s^2).
\end{aligned}
\end{equation}
Here the coefficient in front of \eqref{eq:3.4} cannot be omitted; moreover, the \(\Delta s^2/n\) in the trace-free Hessian cannot be replaced by \(\Delta_f s^2/m\).

\begin{lemma} \label{Lemma 2.2.}One has
\begin{equation}\label{eq:3.7}
\begin{aligned}
&2s^2\left(\Delta_f+\frac{3mk}{2}\right)|\nabla s|^2
+(4-2m)\langle\nabla s^2,\nabla|\nabla s|^2\rangle\\
&=
|\mathring{\operatorname{Hess}}s^2|^2
+D_m
+\operatorname{Ric}_f^m(\nabla s^2,\nabla s^2)
+\frac{mk^2}{4}s^4.
\end{aligned}
\end{equation}
\end{lemma}

\begin{proof}
Since \(|\nabla s^2|^2=4s^2|\nabla s|^2\), the left-hand side of the Bochner formula expands as
\begin{equation}\label{eq:3.8}
\begin{aligned}
\frac12\Delta_f|\nabla s^2|^2
&=2\Delta_f(s^2|\nabla s|^2)\\
&=2s^2\Delta_f|\nabla s|^2
+2|\nabla s|^2\Delta_f s^2
+4\langle\nabla s^2,\nabla|\nabla s|^2\rangle\\
&=2s^2\Delta_f|\nabla s|^2
+4m|\nabla s|^4
-mks^2|\nabla s|^2
+4\langle\nabla s^2,\nabla|\nabla s|^2\rangle.
\end{aligned}
\end{equation}
The gradient term on the right-hand side is
\begin{equation}\label{eq:3.9}
\begin{aligned}
\langle\nabla s^2,\nabla\Delta_f s^2\rangle
=2m\langle\nabla s^2,\nabla|\nabla s|^2\rangle
-\frac{mk}{2}|\nabla s^2|^2
=2m\langle\nabla s^2,\nabla|\nabla s|^2\rangle
-2mks^2|\nabla s|^2.
\end{aligned}
\end{equation}
The square term is
\begin{equation}\label{eq:3.10}
\begin{aligned}
\frac{(\Delta_f s^2)^2}{m}
=\frac1m\left(2m|\nabla s|^2-\frac{mk}{2}s^2\right)^2
=4m|\nabla s|^4-2mks^2|\nabla s|^2+\frac{mk^2}{4}s^4.
\end{aligned}
\end{equation}
Substituting these two formulas into \eqref{eq:3.6} and comparing with \eqref{eq:3.8}, the terms \(4m|\nabla s|^4\) cancel, and we obtain
\begin{equation}\label{eq:3.11}
\begin{aligned}
2s^2\Delta_f|\nabla s|^2
={}&
|\mathring{\operatorname{Hess}}s^2|^2+D_m+\operatorname{Ric}_f^m(\nabla s^2,\nabla s^2)\\
&+(2m-4)\langle\nabla s^2,\nabla|\nabla s|^2\rangle
-3mks^2|\nabla s|^2+\frac{mk^2}{4}s^4.
\end{aligned}
\end{equation}
Moving terms gives \eqref{eq:3.7}.
\end{proof}

\begin{proposition} One has
\begin{equation}\label{eq:3.12}
\begin{aligned}
&\left(\Delta_f-\frac{m(m-2)k}{4}\right)(|\nabla s|^2G)\\
&=
\frac{G}{2s^2}\left[
|\mathring{\operatorname{Hess}}s^2|^2
+D_m
+(\operatorname{Ric}_f^m-(m-1)kg)(\nabla s^2,\nabla s^2)
\right]+\frac{k(m-4)}{2}G|\nabla s|^2
+\frac{mk^2}{8}s^2G.
\end{aligned}
\end{equation}
\end{proposition}
\begin{proof}
First,
\begin{equation}\label{eq:3.13}
\nabla G=\frac{2-m}{2s^2}G\nabla s^2.
\end{equation}
Using the product rule and cancelling the potential term in \(|\nabla s|^2\Delta_f G\), we get
\[
\begin{aligned}
&\left(\Delta_f-\frac{m(m-2)k}{4}\right)(|\nabla s|^2G)\\
&=G\Delta_f|\nabla s|^2+2\langle\nabla G,\nabla|\nabla s|^2\rangle\\
&=\frac{G}{2s^2}\left[
|\mathring{\operatorname{Hess}}s^2|^2
+D_m
+\operatorname{Ric}_f^m(\nabla s^2,\nabla s^2)
\right]\\
&\quad+\frac{m-2}{s^2}G\langle\nabla s^2,\nabla|\nabla s|^2\rangle
-\frac{3mk}{2}G|\nabla s|^2+\frac{mk^2}{8}s^2G+\frac{2-m}{s^2}G\langle\nabla s^2,\nabla|\nabla s|^2\rangle\\
&=\frac{G}{2s^2}\left[
|\mathring{\operatorname{Hess}}s^2|^2
+D_m
+\operatorname{Ric}_f^m(\nabla s^2,\nabla s^2)
\right]
-\frac{3mk}{2}G|\nabla s|^2+\frac{mk^2}{8}s^2G.
\end{aligned}
\]
The coefficients of the cross terms are \((m-2)+(2-m)=0\). Extracting \((m-1)k|\nabla s^2|^2\) from the curvature term and using \(|\nabla s^2|^2=4s^2|\nabla s|^2\), the coefficient of \(G|\nabla s|^2\) becomes \(2(m-1)k-3mk/2=k(m-4)/2\), which gives the conclusion.
\end{proof}

\begin{proposition}One has the local identity
\begin{equation}\label{eq:3.14}
\begin{aligned}
&\left(\Delta_f-\frac{m(m-2)k}{4}\right)
\left((4|\nabla s|^2+ks^2)G\right)\\
&=
8(2s)^{-m}\left[
|\mathring{\operatorname{Hess}}s^2|^2
+D_m
+(\operatorname{Ric}_f^m-(m-1)kg)(\nabla s^2,\nabla s^2)
\right]\ge0.
\end{aligned}
\end{equation}
However, when \(m>n\), one cannot deduce that \(4|\nabla s|^2+ks^2\le1\).
\end{proposition}
\begin{proof}
The other product term expands as
\[
\begin{aligned}
\left(\Delta_f-\frac{m(m-2)k}{4}\right)(ks^2G)
&=kG\Delta_f s^2+2k\langle\nabla s^2,\nabla G\rangle\\
&=kG\left(2m|\nabla s|^2-\frac{mk}{2}s^2\right)
+k(2-m)G\frac{|\nabla s^2|^2}{s^2}\\
&=2k(4-m)G|\nabla s|^2-\frac{mk^2}{2}s^2G.
\end{aligned}
\]
Multiplying \eqref{eq:3.12} by \(4\) and adding, the coefficient of \(G|\nabla s|^2\) is \(2k(m-4)+2k(4-m)=0\), and the coefficient of \(s^2G\) is \(mk^2/2-mk^2/2=0\). Also \(2G/s^2=8(2s)^{-m}\), so \eqref{eq:3.14} follows. However, in \eqref{eq:2.16} we have \(2\alpha-2<0\), hence \(4|\nabla s|^2+ks^2\to+\infty\). At the pole the upper-bound control needed for the maximum principle is not available, so one cannot apply the sharp estimate of the original paper.
\end{proof}

\begin{remark} By the curvature assumption, the three terms in the brackets of \eqref{eq:3.14} are all nonnegative. If their sum is zero, then each term is zero, but this does not provide an already established sharp gradient estimate or its global rigidity. In the independent case \(m=n\) with \(f\) constant, one can directly use the ordinary Bochner formula and set \(D_m=0\); one cannot substitute \(m=n\) into the denominator of \eqref{eq:3.4}. In this case, if \(G\sim c_p\rho^{2-n}\), then \(4|\nabla s|^2+ks^2\to c_p^{-2/(n-2)}\), and in particular the limit is \(1\) when \(c_p=1\).
\end{remark}
\begin{proposition} Define the weighted drift operator corresponding to the original paper by
\begin{equation}\label{eq:3.15}
\begin{aligned}
\mathcal L_f u
&=\Delta_f u+\frac{2-m}{s^2}\langle\nabla s^2,\nabla u\rangle\\
&=\Delta_f u+2\langle\nabla\log G,\nabla u\rangle\\
&=G^{-2}\operatorname{div}_f(G^2\nabla u).
\end{aligned}
\end{equation}
Then
\begin{equation}\label{eq:3.16}
\mathcal L_f(4|\nabla s|^2+ks^2)
=
\frac{2}{s^2}\left[
|\mathring{\operatorname{Hess}}s^2|^2
+D_m
+(\operatorname{Ric}_f^m-(m-1)kg)(\nabla s^2,\nabla s^2)
\right].
\end{equation}
\end{proposition} 
\begin{proof}
Expanding the divergence form gives
\[
G^{-2}\operatorname{div}_f(G^2\nabla u)
=\Delta_f u+2G^{-1}\langle\nabla G,\nabla u\rangle,
\]
and using \eqref{eq:3.13} gives \eqref{eq:3.15}. By Lemma \ref{Lemma 2.2.}, after adding the drift term to \(\Delta_f|\nabla s|^2\), the coefficients \((m-2)+(2-m)\) cancel, leaving
\[
\begin{aligned}
\mathcal L_f|\nabla s|^2
&=\frac{1}{2s^2}\left[
|\mathring{\operatorname{Hess}}s^2|^2
+D_m
+(\operatorname{Ric}_f^m-(m-1)kg)(\nabla s^2,\nabla s^2)
\right]+\frac{(m-4)k}{2}|\nabla s|^2+\frac{mk^2}{8}s^2.
\end{aligned}
\]
On the other hand,
\[
\begin{aligned}
\mathcal L_f s^2
=2m|\nabla s|^2-\frac{mk}{2}s^2+\frac{2-m}{s^2}|\nabla s^2|^2
=2(4-m)|\nabla s|^2-\frac{mk}{2}s^2.
\end{aligned}
\]
Multiplying the former by \(4\) and the latter by \(k\) and adding, the two remaining coefficients are \(2(m-4)k+2(4-m)k=0\) and \(mk^2/2-mk^2/2=0\), which gives \eqref{eq:3.16}. One can also verify independently from
\[
\left(\Delta_f-\frac{m(m-2)k}{4}\right)(Gu)=G\mathcal L_f u.
\]
Here one must use \(\mathcal L_f\), and cannot mistakenly write the left-hand side as \(\Delta_f(4|\nabla s|^2+ks^2)\).
\end{proof}

\section{Three monotonicity formulas}\label{sec:basic}

\subsection{The three integral quantities}

If \(u\in C^2(M\setminus\{p\})\) and \(uG\) is absolutely integrable near the pole with respect to \(e^{-f}dV\), define
\begin{equation}\label{eq:4.1}
\begin{aligned}
I_{f,u}(r)
:={}&
(2\operatorname{sn}_k(r/2))^{1-m}\operatorname{cs}_k(r/2)
\int_{b_f=r}u|\nabla b_f|e^{-f}d\sigma+\frac{mk}{4}\int_{b_f<r}uG e^{-f}dV.
\end{aligned}
\end{equation}
Using
\[
\nabla G=(2-m)(2s)^{1-m}\operatorname{cs}_k(b_f/2)\nabla b_f,
\]
equivalently,
\begin{equation}\label{eq:4.2}
I_{f,u}(r)
=
\frac{1}{m-2}\int_{b_f=r}u|\nabla G|e^{-f}d\sigma
+\frac{mk}{4}\int_{b_f<r}uG e^{-f}dV.
\end{equation}
If \(u(4|\nabla s|^2-ks^2)\) is also absolutely integrable, define
\begin{equation}\label{eq:4.3}
\begin{aligned}
J_{f,u}(r)
:={}&
(2\operatorname{sn}_k(r/2))^{-m}
\int_{b_f<r}u(4|\nabla s|^2-ks^2)e^{-f}dV+\frac{k}{4}\int_{b_f<r}uG e^{-f}dV.
\end{aligned}
\end{equation}
The above curvature-correcting volume integral cannot be omitted. Taking \(u=4|\nabla s|^2+ks^2\), we obtain respectively
\[
A_f=I_{f,4|\nabla s|^2+ks^2},\qquad
V_f=J_{f,4|\nabla s|^2+ks^2}.
\]
The latter is used only when its integral converges. 
All differentiations are first performed at regular values and then understood as almost-everywhere identities via locally absolutely continuous representatives. The reason is as follows: in a compact environment away from the pole, the flux difference is written as a volume integral by the weighted divergence formula. At critical points of \(s\), \eqref{eq:3.1} gives \(\Delta_f s=-mks/4\ne0\), so the critical set has measure zero; otherwise, at an almost-everywhere density point of the zero set of \(\nabla s\), one would have \(\operatorname{Hess}s=0\), contradicting this identity. By the coarea formula, the flux and the volume integral are locally absolutely continuous in the parameter. Hence the sign of the derivative indeed implies monotonicity over the entire parameter interval.

\begin{lemma}\label{Lemma3.1}
 Under the above integrability conditions,
\begin{equation}\label{eq:4.6}
I_{f,u}'(r)
=
(2\operatorname{sn}_k(r/2))^{1-m}\operatorname{cs}_k(r/2)
\int_{b_f=r}\langle\nabla u,\nu\rangle e^{-f}d\sigma.
\end{equation}
If in addition \(\Delta_f u\) is absolutely integrable and the following limit exists, then
\begin{equation}\label{eq:4.7}
\begin{aligned}
I_{f,u}'(r)
=
(2\operatorname{sn}_k(r/2))^{1-m}\operatorname{cs}_k(r/2)\left[
\int_{b_f<r}\Delta_f u\,e^{-f}dV
+\lim_{\epsilon\downarrow0}\int_{\rho=\epsilon}\langle\nabla u,\nabla\rho\rangle e^{-f}d\sigma
\right].
\end{aligned}
\end{equation}
Only when this limit is zero does one obtain a formula without the inner flux term.
\end{lemma}
\begin{proof}
Write the boundary term in \eqref{eq:4.2} as
\[
-(m-2)^{-1}\int_{b_f=r}u\langle\nabla G,\nu\rangle e^{-f}d\sigma.
\]
The divergence formula and the coarea formula give
\[
\begin{aligned}
&\frac{d}{dr}\left(\frac{1}{m-2}\int_{b_f=r}u|\nabla G|e^{-f}d\sigma\right)\\
&=\frac{1}{m-2}\int_{b_f=r}
\frac{-u\Delta_f G-\langle\nabla u,\nabla G\rangle}{|\nabla b_f|}
e^{-f}d\sigma\\
&=-\frac{mk}{4}\int_{b_f=r}\frac{uG}{|\nabla b_f|}e^{-f}d\sigma
+(2\operatorname{sn}_k(r/2))^{1-m}\operatorname{cs}_k(r/2)
\int_{b_f=r}\langle\nabla u,\nu\rangle e^{-f}d\sigma.
\end{aligned}
\]
The derivative of the volume integral is exactly the negative of the first term, which yields \eqref{eq:4.6}. On the other hand,
\[
\int_{\{b_f<r\}\setminus B_\epsilon(p)}\Delta_f u\,e^{-f}dV
=
\int_{b_f=r}\langle\nabla u,\nu\rangle e^{-f}d\sigma
-\int_{\rho=\epsilon}\langle\nabla u,\nabla\rho\rangle e^{-f}d\sigma.
\]
Taking the limit gives \eqref{eq:4.7}. For example, \(|\nabla u|=O(\rho^{-\gamma})\) with \(\gamma<n-1\) is sufficient to guarantee that the inner flux tends to zero, but the integrability of \(\Delta_f u\) must still be checked.
\end{proof}

\subsection{Pole flux for continuous functions: supplementary argument}

If \(u\) is continuous at \(p\), then by the flux normalization of the Green function and the integrability of \(uG\),
\begin{equation}\label{eq:4.8}
\lim_{r\downarrow0}I_{f,u}(r)=\frac{\kappa}{m-2}u(p).
\end{equation}
More precisely, the boundary term can be written as
\[
\frac{u(p)}{m-2}\int_{b_f=r}|\nabla G|e^{-f}d\sigma
+\frac{1}{m-2}\int_{b_f=r}(u-u(p))|\nabla G|e^{-f}d\sigma.
\]
The first term tends to \(\kappa u(p)/(m-2)\), and the absolute value of the second term is bounded above by
\[
\sup_{b_f=r}|u-u(p)|\int_{b_f=r}|\nabla G|e^{-f}d\sigma/(m-2),
\]
so it tends to zero. The volume integral term also tends to zero.

If in addition \(\Delta_f u\in L^1(e^{-f}dV)\), then one can prove that the inner flux in Lemma \ref{Lemma3.1} is zero without first assuming a pointwise upper bound on the gradient growth. By the derivative asymptotics of \(s\), all sufficiently small positive level values are regular. Let \(a=2\operatorname{sn}_k(r/2)\), and temporarily regard \(I_{f,u}\) as a function of \(a\). A function \(F\) that has a finite limit at zero and is differentiable on the positive half-line admits a sequence \(a_i\downarrow0\) such that \(a_iF'(a_i)\to0\). Indeed, for any \(\delta_i\downarrow0\), applying the mean value theorem on \([\delta_i/2,\delta_i]\) gives some \(a_i\) such that
\[
|a_iF'(a_i)|\le 2|F(\delta_i)-F(\delta_i/2)|\to0.
\]
By \eqref{eq:4.6} and \(da/dr=\operatorname{cs}_k(r/2)\),
\[
\int_{2s=a_i}\langle\nabla u,\nu\rangle e^{-f}d\sigma
=
a_i^{m-1}\frac{dI_{f,u}}{da}(a_i)
=
a_i^{m-2}\left(a_i\frac{dI_{f,u}}{da}(a_i)\right)\longrightarrow0.
\]
Using the divergence formula on \(\{a_i<2s<2\operatorname{sn}_k(r/2)\}\) and then taking the limit by the absolute integrability of \(\Delta_f u\), we obtain
\[
\int_{b_f=r}\langle\nabla u,\nu\rangle e^{-f}d\sigma
=
\int_{b_f<r}\Delta_f u\,e^{-f}dV.
\]
Applying the divergence formula further on the region obtained by removing a small geodesic ball, we also find that the flux inside the geodesic ball tends to zero. Here continuity provides the finite initial value, while absolute integrability ensures that the limit is independent of the cutoff sequence; their roles are different. For the specific functions in Sections 3 and 4 that diverge at \(p\), we instead check the flux using their explicit asymptotic orders, and do not apply this continuity argument.

\begin{lemma}\label{Lemma 3.2} Under the conditions of \eqref{eq:4.1},
\begin{equation}\label{eq:4.9}
\begin{aligned}
&\frac{d}{dr}\left((2\operatorname{sn}_k(r/2))^{2-m}I_{f,u}(r)\right)\\
&=
(2\operatorname{sn}_k(r/2))^{1-m}\operatorname{cs}_k(r/2)
\left[
\int_{b_f=r}\langle\nabla(uG),\nu\rangle e^{-f}d\sigma
-\frac{m(m-2)k}{4}\int_{b_f<r}uG e^{-f}dV
\right].
\end{aligned}
\end{equation}
\end{lemma}
\begin{proof}
The product rule and \eqref{eq:4.6} turn the left-hand side into
\[
\begin{aligned}
&(2\operatorname{sn}_k(r/2))^{1-m}\operatorname{cs}_k(r/2)\left[
(2-m)I_{f,u}(r)
+(2\operatorname{sn}_k(r/2))^{2-m}
\int_{b_f=r}\langle\nabla u,\nu\rangle e^{-f}d\sigma
\right].
\end{aligned}
\]
Substituting \eqref{eq:4.2} into the bracket, using
\[
\langle\nabla G,\nu\rangle=-|\nabla G|,\qquad
G|_{b_f=r}=(2\operatorname{sn}_k(r/2))^{2-m},
\]
the boundary terms combine as
\[
u\langle\nabla G,\nu\rangle+G\langle\nabla u,\nu\rangle
=\langle\nabla(uG),\nu\rangle,
\]
and the volume integral coefficient is \((2-m)mk/4=-m(m-2)k/4\), which gives the conclusion. This step does not use the divergence formula for \(uG\) at the pole.
\end{proof}

\begin{remark}\label{remark3.3}
For all \(m>n\ge3\), \(A_f(r)\) is finite for every admissible positive \(r\), but
\begin{equation}\label{eq:4.10}
A_f(r)\asymp
(2\operatorname{sn}_k(r/2))^{-2(m-n)/(n-2)}
\longrightarrow+\infty.
\end{equation}
\end{remark}
\begin{proof}
By \eqref{eq:2.11} and \eqref{eq:2.16}, the radial order of
\((4|\nabla s|^2+ks^2)G e^{-f}dV\) is \(\rho^{2\alpha-1}d\rho\), so it is integrable and its small-ball integral is \(O(\rho^{2\alpha})\). On the boundary,
\[
\operatorname{cs}_k(r/2)|\nabla b_f|=2|\nabla s|\asymp\rho^{\alpha-1},
\]
so the boundary term of \(A_f\) has order
\[
\rho^{\alpha(1-m)}\rho^{2\alpha-2}\rho^{\alpha-1}\rho^{n-1}
=
\rho^{2\alpha-2},
\]
where we used \(n-2=\alpha(m-2)\). Since \(2\operatorname{sn}_k(r/2)\asymp\rho^\alpha\), we obtain \eqref{eq:4.10}. The leading coefficients are positive, so this two-sided order estimate indeed gives divergence.
\end{proof}

\begin{theorem}\label{Theorem 3.4}For regular values \(r_1<r_2\),
\begin{equation}\label{eq:4.11}
\begin{aligned}
&\left[
\frac{(2\operatorname{sn}_k(r/2))^{m-1}}{\operatorname{cs}_k(r/2)}
\frac{d}{dr}\left((2\operatorname{sn}_k(r/2))^{2-m}A_f(r)\right)
\right]_{r=r_1}^{r=r_2}\\
&=
8\int_{r_1<b_f<r_2}
(2s)^{-m}\left[
|\mathring{\operatorname{Hess}}s^2|^2
+D_m
+(\operatorname{Ric}_f^m-(m-1)kg)(\nabla s^2,\nabla s^2)
\right]e^{-f}dV.
\end{aligned}
\end{equation}
The right-hand side is nonnegative; this formula does not claim that
\((2\operatorname{sn}_k(r/2))^{2-m}(A_f-C_0)\) is nondecreasing.
\end{theorem}
\begin{proof}
Take \(u=4|\nabla s|^2+ks^2\) in Lemma \ref{Lemma 3.2} and subtract the identities at two regular level values. The boundary flux difference minus the potential volume integral equals the annular integral
\[
\int_{r_1<b_f<r_2}
\left(\Delta_f-\frac{m(m-2)k}{4}\right)
\left((4|\nabla s|^2+ks^2)G\right)e^{-f}dV.
\]
The annulus is away from the pole, so the divergence formula is legitimate, and substituting \eqref{eq:3.14} gives \eqref{eq:4.11}. We do not take \(r_1\) to zero, so no unverified finite inner flux assumption is included.
\end{proof}

\begin{proposition}\label{Proposition 3.5.} If \(u\in C^2(M\setminus\{p\})\) and \(uG\) is absolutely integrable, then
\begin{equation}\label{eq:4.12}
\frac{(2\operatorname{sn}_k(r/2))^{3-m}}{\operatorname{cs}_k(r/2)}
I_{f,u}'(r)
=
-\int_{s>\operatorname{sn}_k(r/2)}G^2\mathcal L_f u\,e^{-f}dV.
\end{equation}
No continuity of \(u\) at the pole is required.
\end{proposition} 
\begin{proof}
The exterior region is compact and away from the pole, and its outward normal is \(-\nu\). By \eqref{eq:3.15} and the divergence formula,
\[
\begin{aligned}
&\int_{s>\operatorname{sn}_k(r/2)}G^2\mathcal L_f u\,e^{-f}dV\\
&=\int_{s>\operatorname{sn}_k(r/2)}\operatorname{div}_f(G^2\nabla u)e^{-f}dV\\
&=-(2\operatorname{sn}_k(r/2))^{4-2m}
\int_{b_f=r}\langle\nabla u,\nu\rangle e^{-f}d\sigma\\
&=-\frac{(2\operatorname{sn}_k(r/2))^{3-m}}{\operatorname{cs}_k(r/2)}I_{f,u}'(r).
\end{aligned}
\]
This proof does not pass through the pole and does not need to assert that
\(\int_M\Delta_f u\,e^{-f}dV=0\).
\end{proof}

\subsection{Three monotonicity formulas}

\begin{theorem} \label{Theorem 3.6.} For all \(m>n\ge3\), \(A_f\) is locally absolutely continuous and satisfies
\begin{equation}\label{eq:4.13}
\begin{aligned}
&\frac{(2\operatorname{sn}_k(r/2))^{3-m}}{\operatorname{cs}_k(r/2)}A_f'(r)\\
&=
-8\int_{s>\operatorname{sn}_k(r/2)}
(2s)^{2-2m}\left[
|\mathring{\operatorname{Hess}}s^2|^2
+D_m
+(\operatorname{Ric}_f^m-(m-1)kg)(\nabla s^2,\nabla s^2)
\right]e^{-f}dV\le0.
\end{aligned}
\end{equation}
Therefore \(A_f\) is nonincreasing, but its limit at zero is \(+\infty\).
\end{theorem}
\begin{proof}
Remark  \ref{remark3.3} guarantees that the defining integrals are finite. Taking
\(u=4|\nabla s|^2+ks^2\) in Proposition \ref{Proposition 3.5.} and using \eqref{eq:3.16} and
\(G^2(2/s^2)=8(2s)^{2-2m}\), we obtain the formula. The exterior integral is away from the pole and is always finite; the curvature assumption and \(D_m\ge0\) give the sign.
\end{proof}

This also rigorously excludes the finite-constant corrected version of Theorem \ref{Theorem 3.4}: for any finite \(C_0\), for sufficiently small \(r\) one has \(A_f(r)>C_0\), hence
\[
\begin{aligned}
&\frac{d}{dr}\left((2\operatorname{sn}_k(r/2))^{2-m}(A_f-C_0)\right)\\
&=(2-m)(2\operatorname{sn}_k(r/2))^{1-m}\operatorname{cs}_k(r/2)(A_f-C_0)
+(2\operatorname{sn}_k(r/2))^{2-m}A_f'<0
\end{aligned}
\]
almost everywhere. This is not a nondecreasing quantity.

\begin{remark} The product
\((4|\nabla s|^2+ks^2)(4|\nabla s|^2-ks^2)\) in the definition has a strictly positive leading order \(\rho^{4\alpha-4}\) near the pole. Its volume integral converges if and only if
\begin{equation}\label{eq:4.14}
n+4\alpha-4>0,
\qquad\text{i.e.}\qquad
\begin{cases}
n\ge4,\ m>n,\\
n=3,\ 3<m<6.
\end{cases}
\end{equation}
When \(n=3,\ m=6\), the divergence is logarithmic; when \(n=3,\ m>6\), it is a power divergence. In the convergence range, radial integration gives
\begin{equation}\label{eq:4.15}
V_f(r)\asymp
(2\operatorname{sn}_k(r/2))^{-2(m-n)/(n-2)},
\qquad
(2\operatorname{sn}_k(r/2))^m V_f(r)\longrightarrow0.
\end{equation}
The second limit can also be obtained directly from the small-region limit of the two convergent volume integrals in the definition.
\end{remark}
\begin{proposition}\label{Proposition 3.8.}
 When the defining integrals of \(I_{f,u}\) and \(J_{f,u}\) are absolutely convergent,
\begin{equation}\label{eq:4.16}
J_{f,u}'(r)
=
\frac{1}{2\tan_k(r/2)}\left(I_{f,u}(r)-mJ_{f,u}(r)\right).
\end{equation}
\end{proposition}
\begin{proof}
Using the product rule and the coarea formula directly,
\[
\begin{aligned}
J_{f,u}'(r)
={}&
-m(2\operatorname{sn}_k(r/2))^{-m-1}\operatorname{cs}_k(r/2)
\int_{b_f<r}u(4|\nabla s|^2-ks^2)e^{-f}dV\\
&+(2\operatorname{sn}_k(r/2))^{-m}
\int_{b_f=r}\frac{u(4|\nabla s|^2-ks^2)}{|\nabla b_f|}e^{-f}d\sigma+\frac{k}{4}\int_{b_f=r}\frac{uG}{|\nabla b_f|}e^{-f}d\sigma.
\end{aligned}
\]
On the boundary \(G=(2\operatorname{sn}_k(r/2))^{2-m}\), so the \(-ks^2\) part in the last two terms and the \(kG/4\) part cancel; also
\(4|\nabla s|^2=\operatorname{cs}_k^2(r/2)|\nabla b_f|^2\). Hence
\[
\begin{aligned}
J_{f,u}'(r)
={}&
-m\frac{\operatorname{cs}_k(r/2)}{2\operatorname{sn}_k(r/2)}
(2\operatorname{sn}_k(r/2))^{-m}
\int_{b_f<r}u(4|\nabla s|^2-ks^2)e^{-f}dV\\
&+(2\operatorname{sn}_k(r/2))^{-m}\operatorname{cs}_k^2(r/2)
\int_{b_f=r}u|\nabla b_f|e^{-f}d\sigma\\
={}&
\frac{\operatorname{cs}_k(r/2)}{2\operatorname{sn}_k(r/2)}
\left(I_{f,u}(r)-mJ_{f,u}(r)\right).
\end{aligned}
\]
In the last step, the coefficient of \(\int uG e^{-f}dV\) is \(mk/4-mk/4=0\). Since
\(\operatorname{cs}_k/(2\operatorname{sn}_k)=1/(2\tan_k)\), the conclusion follows. No integration by parts for \(u\) at the pole is used here.
\end{proof}

\begin{corollary} With the normalization \eqref{eq:2.5},
\begin{equation}\label{eq:4.17}
I_{f,1}\equiv\frac{\kappa}{m-2},
\qquad
J_{f,1}\equiv\frac{\kappa}{m(m-2)}.
\end{equation}
\end{corollary} 
\begin{proof}
By \eqref{eq:4.6}, \(I_{f,1}'=0\); by \eqref{eq:2.13} and the integrability of \(G\),
\(\lim_{r\downarrow0}I_{f,1}(r)=\kappa/(m-2)\). Substituting this constant into \eqref{eq:4.16} and multiplying by the integrating factor \((2\operatorname{sn}_k(r/2))^m\), we get
\[
\frac{d}{dr}\left((2\operatorname{sn}_k(r/2))^mJ_{f,1}(r)\right)
=
\frac{\kappa}{m-2}(2\operatorname{sn}_k(r/2))^{m-1}\operatorname{cs}_k(r/2).
\]
After integration,
\[
J_{f,1}=\frac{\kappa}{m(m-2)}
+C(2\operatorname{sn}_k(r/2))^{-m}.
\]
Since \(4|\nabla s|^2-ks^2\) and \(G\) are integrable, the definition gives
\((2\operatorname{sn}_k(r/2))^mJ_{f,1}(r)\to0\), so \(C=0\).
\end{proof}

The constant is determined by \(\kappa\) in the weak normalization, not by the small-ball area of the effective dimension \(m\).

\begin{lemma}\label{Lemma 3.10.} On the punctured manifold,
\begin{equation}\label{eq:4.18}
\begin{aligned}
\frac12(\Delta_f+2mk)|\nabla s^2|^2
&=
|\mathring{\operatorname{Hess}}s^2|^2
+D_m
+\operatorname{Ric}_f^m(\nabla s^2,\nabla s^2)
+\frac{(\Delta_f s^2)^2}{m}
+\langle\nabla(\Delta_f+mk)s^2,\nabla s^2\rangle.
\end{aligned}
\end{equation}
Moreover,
\begin{equation}\label{eq:4.19}
\begin{aligned}
\frac{(\Delta_f s^2)^2}{m}
={}&
\frac m4(4|\nabla s|^2+ks^2)(4|\nabla s|^2-ks^2)-\frac{mk}{2}|\nabla s^2|^2+\frac{mk^2}{2}s^4.
\end{aligned}
\end{equation}
If \eqref{eq:4.14} holds, then one has the integral identity
\begin{equation}\label{eq:4.20}
\begin{aligned}
&\frac12\int_{b_f<r}\Delta_f|\nabla s^2|^2e^{-f}dV\\
&=
\int_{b_f<r}\left[
|\mathring{\operatorname{Hess}}s^2|^2
+D_m
+(\operatorname{Ric}_f^m-(m-1)kg)(\nabla s^2,\nabla s^2)
\right]e^{-f}dV\\
&\quad+\frac m2\int_{b_f=r}(4|\nabla s|^2+ks^2)\langle\nabla s^2,\nu\rangle e^{-f}d\sigma\\
&\quad-\frac{m(m-1)}{4}\int_{b_f<r}
(4|\nabla s|^2+ks^2)(4|\nabla s|^2-ks^2)e^{-f}dV\\
&\quad-\frac{(m+2)k}{2}\int_{b_f<r}|\nabla s^2|^2e^{-f}dV
+\frac{mk^2}{2}\int_{b_f<r}s^4e^{-f}dV.
\end{aligned}
\end{equation}
\end{lemma}
\begin{proof}
Adding \(mk|\nabla s^2|^2\) to both sides of \eqref{eq:3.6} and absorbing this term into the gradient term gives \eqref{eq:4.18}. On the other hand, \eqref{eq:3.2} gives
\[
\frac{(\Delta_f s^2)^2}{m}
=
\frac m4(4|\nabla s|^2-ks^2)^2
=
4m|\nabla s|^4-2mks^2|\nabla s|^2+\frac{mk^2}{4}s^4.
\]
Expanding the right-hand side of \eqref{eq:4.19} and using \(|\nabla s^2|^2=4s^2|\nabla s|^2\), one obtains the same expression.

For integration, first remove \(B_\epsilon(p)\). By \eqref{eq:3.3},
\[
\begin{aligned}
&\int_{\{b_f<r\}\setminus B_\epsilon(p)}\langle\nabla(\Delta_f+mk)s^2,\nabla s^2\rangle e^{-f}dV\\
&=\frac m2\int_{b_f=r}(4|\nabla s|^2+ks^2)\langle\nabla s^2,\nu\rangle e^{-f}d\sigma\\
&\quad-\frac m2\int_{\rho=\epsilon}(4|\nabla s|^2+ks^2)\langle\nabla s^2,\nabla\rho\rangle e^{-f}d\sigma\\
&\quad-\frac{m^2}{4}\int_{\{b_f<r\}\setminus B_\epsilon(p)}
(4|\nabla s|^2+ks^2)(4|\nabla s|^2-ks^2)e^{-f}dV.
\end{aligned}
\]
Writing the curvature term as the curvature excess term plus \((m-1)k|\nabla s^2|^2\), and substituting \eqref{eq:4.19}, the coefficient of the product term is \(m/4-m^2/4=-m(m-1)/4\); after moving the \(mk|\nabla s^2|^2\) term from the left-hand side, the coefficient of \(|\nabla s^2|^2\) is \((m-1)k-mk/2-mk=-(m+2)k/2\). This gives the truncated form of \eqref{eq:4.20}, with the above inner boundary term.

Now we verify the limit. By the second-order asymptotics,
\[
\begin{aligned}
&|\mathring{\operatorname{Hess}}s^2|^2+D_m
+(\operatorname{Ric}_f^m-(m-1)kg)(\nabla s^2,\nabla s^2)
=O(\rho^{4\alpha-4}),\\
&(4|\nabla s|^2+ks^2)|\nabla s^2|+|\nabla|\nabla s^2|^2|
=O(\rho^{4\alpha-3}).
\end{aligned}
\]
Hence the inner boundary integral is \(O(\epsilon^{n+4\alpha-4})\to0\), and the worst radial order of the volume integral is \(O(\rho^{n+4\alpha-5}d\rho)\), which is integrable under \eqref{eq:4.14}. The integrability of \(\Delta_f|\nabla s^2|^2\) is judged from \eqref{eq:4.18}: therein
\(\nabla(\Delta_f+mk)s^2=(m/2)\nabla(4|\nabla s|^2+ks^2)\), whose worst order is still \(O(\rho^{4\alpha-4})\). Taking the limit gives \eqref{eq:4.20}. No fourth-order derivative estimate is inferred arbitrarily from the second-order asymptotics.

\end{proof}

\begin{theorem}\label{Theorem 3.11.} If \eqref{eq:4.14} holds, then \(V_f\) is finite and locally absolutely continuous,
\begin{equation}\label{eq:4.21}
V_f'(r)=\frac{1}{2\tan_k(r/2)}(A_f(r)-mV_f(r))\le0,
\end{equation}
and
\begin{equation}\label{eq:4.22}
\begin{aligned}
&A_f'(r)-2(m-1)V_f'(r)\\
&=
\frac{8\operatorname{cs}_k(r/2)}{(2\operatorname{sn}_k(r/2))^{m+1}}
\int_{b_f<r}\left[
|\mathring{\operatorname{Hess}}s^2|^2
+D_m
+(\operatorname{Ric}_f^m-(m-1)kg)(\nabla s^2,\nabla s^2)
\right]e^{-f}dV\ge0.
\end{aligned}
\end{equation}
Therefore \(V_f\) is nonincreasing and \(A_f-2(m-1)V_f\) is nondecreasing.
\end{theorem}

\begin{proof}
To directly match the cancellation in Lemma \ref{Lemma 3.10.} with \(A_f,V_f\), multiply the core equation by \(s^2\) and rearrange it by the product rule as
\begin{equation}\label{eq:4.23}
\begin{aligned}
&\operatorname{div}_f\left(
s^2\nabla(4|\nabla s|^2+ks^2)
-(m-1)(4|\nabla s|^2+ks^2)\nabla s^2
\right)\\
&=
s^2\Delta_f(4|\nabla s|^2+ks^2)
+(2-m)\langle\nabla s^2,\nabla(4|\nabla s|^2+ks^2)\rangle\\
&\quad-(m-1)(4|\nabla s|^2+ks^2)\Delta_f s^2\\
&=
2\left[
|\mathring{\operatorname{Hess}}s^2|^2
+D_m
+(\operatorname{Ric}_f^m-(m-1)kg)(\nabla s^2,\nabla s^2)
\right]\\
&\quad-\frac{m(m-1)}{2}
(4|\nabla s|^2+ks^2)(4|\nabla s|^2-ks^2).
\end{aligned}
\end{equation}
The coefficient of the cross term is \(1-(m-1)=2-m\). This divergence rearrangement is compatible with Lemma \ref{Lemma 3.10.} because
\[
\begin{aligned}
&s^2\nabla(4|\nabla s|^2+ks^2)
-(m-1)(4|\nabla s|^2+ks^2)\nabla s^2\\
&=\nabla|\nabla s^2|^2
-m(4|\nabla s|^2+ks^2)\nabla s^2
+2ks^2\nabla s^2.
\end{aligned}
\]
Integrate \eqref{eq:4.23} over the region obtained by removing a small ball. The absolute value of the inner flux is \(O(\epsilon^{n+4\alpha-4})\), which tends to zero under \eqref{eq:4.14}. On the outer boundary,
\(s^2=\operatorname{sn}_k^2(r/2)\) and
\(\langle\nabla s^2,\nu\rangle=\operatorname{sn}_k(r/2)\operatorname{cs}_k(r/2)|\nabla b_f|\). By \eqref{eq:4.6}, the first outer boundary term is
\[
\begin{aligned}
&\operatorname{sn}_k^2(r/2)
\int_{b_f=r}\langle\nabla(4|\nabla s|^2+ks^2),\nu\rangle e^{-f}d\sigma=\frac{(2\operatorname{sn}_k(r/2))^{m+1}}{4\operatorname{cs}_k(r/2)}A_f'(r).
\end{aligned}
\]
By the definition of \(A_f\), the second term is
\[
\begin{aligned}
&-(m-1)\operatorname{sn}_k(r/2)\operatorname{cs}_k(r/2)
\int_{b_f=r}(4|\nabla s|^2+ks^2)|\nabla b_f|e^{-f}d\sigma\\
&=-\frac{m-1}{2}(2\operatorname{sn}_k(r/2))^m
\left[
A_f(r)-\frac{mk}{4}\int_{b_f<r}(4|\nabla s|^2+ks^2)G e^{-f}dV
\right].
\end{aligned}
\]
On the other hand, by the definition of \(V_f\), the product volume integral moved to the other side is
\[
\begin{aligned}
&\frac{m(m-1)}{2}\int_{b_f<r}
(4|\nabla s|^2+ks^2)(4|\nabla s|^2-ks^2)e^{-f}dV\\
&=\frac{m(m-1)}{2}(2\operatorname{sn}_k(r/2))^m
\left[
V_f(r)-\frac{k}{4}\int_{b_f<r}(4|\nabla s|^2+ks^2)G e^{-f}dV
\right].
\end{aligned}
\]
The two volume integrals involving \(G\) have opposite coefficients, both equal to
\[m(m-1)k(2\operatorname{sn}_k(r/2))^m/8.\] Therefore the integrated identity finally becomes
\[
\begin{aligned}
&2\int_{b_f<r}\left[
|\mathring{\operatorname{Hess}}s^2|^2
+D_m
+(\operatorname{Ric}_f^m-(m-1)kg)(\nabla s^2,\nabla s^2)
\right]e^{-f}dV\\
&=\frac{(2\operatorname{sn}_k(r/2))^{m+1}}{4\operatorname{cs}_k(r/2)}A_f'
-\frac{m-1}{2}(2\operatorname{sn}_k(r/2))^m(A_f-mV_f)\\
&=\frac{(2\operatorname{sn}_k(r/2))^{m+1}}{4\operatorname{cs}_k(r/2)}
\left(A_f'-2(m-1)V_f'\right).
\end{aligned}
\]
The last step uses \eqref{eq:4.16}, yielding \eqref{eq:4.22}.

From \(A_f'\le0\) and \eqref{eq:4.22} we already have
\(2(m-1)V_f'\le A_f'\le0\). Next we give an independent verification of \(A_f\le mV_f\) by an integral average. By \eqref{eq:4.16},
\[
\frac{d}{dr}\left((2\operatorname{sn}_k(r/2))^mV_f(r)\right)
=
(2\operatorname{sn}_k(r/2))^{m-1}\operatorname{cs}_k(r/2)A_f(r).
\]
By the zero-point limit \eqref{eq:4.15} and the integrable order in \eqref{eq:4.10},
\begin{equation}\label{eq:4.24}
V_f(r)
=
(2\operatorname{sn}_k(r/2))^{-m}
\int_0^r(2\operatorname{sn}_k(\tau/2))^{m-1}\operatorname{cs}_k(\tau/2)A_f(\tau)d\tau.
\end{equation}
Theorem \ref{Theorem 3.6.} gives \(A_f(\tau)\ge A_f(r)\). The weight is positive, and
\[
\int_0^r(2\operatorname{sn}_k(\tau/2))^{m-1}\operatorname{cs}_k(\tau/2)d\tau
=
\frac1m(2\operatorname{sn}_k(r/2))^m.
\]
Therefore \(mV_f(r)\ge A_f(r)\), and substituting this into \eqref{eq:4.16} gives \eqref{eq:4.21}.
\end{proof}

\subsection{A second proof of Theorem \ref{Theorem 3.11.}}

To exhibit the computation term by term in line with the Bochner expansion, we now derive \eqref{eq:4.22} directly from Lemma \ref{Lemma 3.10.} without using the divergence formula \eqref{eq:4.23}. Throughout this subsection we assume \eqref{eq:4.14}.

First, multiply \eqref{eq:4.20} by
\(8\operatorname{cs}_k(r/2)/(2\operatorname{sn}_k(r/2))^{m+1}\). For the outer boundary term we use
\[
\langle\nabla s^2,\nu\rangle
=
\operatorname{sn}_k(r/2)\operatorname{cs}_k(r/2)|\nabla b_f|,
\qquad
\operatorname{ct}_k(r/2)=\frac{\operatorname{cs}_k(r/2)}{\operatorname{sn}_k(r/2)},
\]
which becomes
\[
\begin{aligned}
&\frac{4m\operatorname{cs}_k(r/2)}
{(2\operatorname{sn}_k(r/2))^{m+1}}
\int_{b_f=r}(4|\nabla s|^2+ks^2)\langle\nabla s^2,\nu\rangle e^{-f}d\sigma\\
&=m\operatorname{ct}_k(r/2)(2\operatorname{sn}_k(r/2))^{1-m}\operatorname{cs}_k(r/2)
\int_{b_f=r}(4|\nabla s|^2+ks^2)|\nabla b_f|e^{-f}d\sigma\\
&=m\operatorname{ct}_k(r/2)A_f(r)
-\frac{m^2k}{4}\operatorname{ct}_k(r/2)
\int_{b_f<r}(4|\nabla s|^2+ks^2)G e^{-f}dV.
\end{aligned}
\]
The product volume integral term becomes
\[
\begin{aligned}
&-\frac{2m(m-1)\operatorname{cs}_k(r/2)}
{(2\operatorname{sn}_k(r/2))^{m+1}}
\int_{b_f<r}(4|\nabla s|^2+ks^2)(4|\nabla s|^2-ks^2)e^{-f}dV\\
&=-m(m-1)\operatorname{ct}_k(r/2)V_f(r)
+\frac{m(m-1)k}{4}\operatorname{ct}_k(r/2)
\int_{b_f<r}(4|\nabla s|^2+ks^2)G e^{-f}dV.
\end{aligned}
\]
The sum of the two coefficients of the volume integrals involving \(G\) is
\[
-\frac{m^2k}{4}+\frac{m(m-1)k}{4}
=
-\frac{mk}{4}.
\]
Also,
\[
\frac{4\operatorname{cs}_k(r/2)}{(2\operatorname{sn}_k(r/2))^{m+1}}
=
2\operatorname{ct}_k(r/2)(2\operatorname{sn}_k(r/2))^{-m}.
\]
Thus the complete form of Lemma \ref{Lemma 3.10.} is
\begin{equation}\label{eq:4.25}
\begin{aligned}
&\frac{4\operatorname{cs}_k(r/2)}{(2\operatorname{sn}_k(r/2))^{m+1}}
\int_{b_f<r}\Delta_f|\nabla s^2|^2e^{-f}dV
+\frac{mk}{4}\operatorname{ct}_k(r/2)
\int_{b_f<r}(4|\nabla s|^2+ks^2)G e^{-f}dV\\
&=
\frac{8\operatorname{cs}_k(r/2)}{(2\operatorname{sn}_k(r/2))^{m+1}}
\int_{b_f<r}\left[
|\mathring{\operatorname{Hess}}s^2|^2
+D_m
+(\operatorname{Ric}_f^m-(m-1)kg)(\nabla s^2,\nabla s^2)
\right]e^{-f}dV\\
&\quad+m\operatorname{ct}_k(r/2)A_f(r)
-m(m-1)\operatorname{ct}_k(r/2)V_f(r)\\
&\quad+2mk^2\operatorname{ct}_k(r/2)(2\operatorname{sn}_k(r/2))^{-m}
\int_{b_f<r}s^4e^{-f}dV\\
&\quad-2(m+2)k\operatorname{ct}_k(r/2)(2\operatorname{sn}_k(r/2))^{-m}
\int_{b_f<r}|\nabla s^2|^2e^{-f}dV.
\end{aligned}
\end{equation}
Next we compute \(A_f'+\operatorname{ct}_k(r/2)A_f\). Using the identity
\[
s^2(4|\nabla s|^2+ks^2)=|\nabla s^2|^2+ks^4,
\]
and differentiating in the outer normal direction, we get
\[
\begin{aligned}
\operatorname{sn}_k^2(r/2)\langle\nabla(4|\nabla s|^2+ks^2),\nu\rangle
=\langle\nabla(|\nabla s^2|^2+ks^4),\nu\rangle
-(4|\nabla s|^2+ks^2)\langle\nabla s^2,\nu\rangle.
\end{aligned}
\]
Multiplying by
\((2\operatorname{sn}_k(r/2))^{1-m}\operatorname{cs}_k(r/2)/\operatorname{sn}_k^2(r/2)\) and integrating, by \eqref{eq:4.6} and the definition of \(A_f\) we obtain
\begin{equation}\label{eq:4.26}
\begin{aligned}
A_f'(r)+\operatorname{ct}_k(r/2)A_f(r)
&=
\frac{4\operatorname{cs}_k(r/2)}{(2\operatorname{sn}_k(r/2))^{m+1}}
\int_{b_f=r}\langle\nabla(|\nabla s^2|^2+ks^4),\nu\rangle e^{-f}d\sigma\\
&\quad+\frac{mk}{4}\operatorname{ct}_k(r/2)
\int_{b_f<r}(4|\nabla s|^2+ks^2)G e^{-f}dV.
\end{aligned}
\end{equation}
The inner boundary terms in the integration by parts satisfy
\[
\int_{\rho=\epsilon}|\nabla|\nabla s^2|^2|e^{-f}d\sigma
=O(\epsilon^{n+4\alpha-4}),
\qquad
\int_{\rho=\epsilon}|\nabla s^4|e^{-f}d\sigma
=O(\epsilon^{n+4\alpha-2}).
\]
Hence both tend to zero, and the boundary integral may be replaced by the interior weighted Laplace integral:
\begin{equation}\label{eq:4.27}
\begin{aligned}
A_f'+\operatorname{ct}_k(r/2)A_f
&=
\frac{4\operatorname{cs}_k(r/2)}{(2\operatorname{sn}_k(r/2))^{m+1}}
\int_{b_f<r}\Delta_f(|\nabla s^2|^2+ks^4)e^{-f}dV\\
&\quad+\frac{mk}{4}\operatorname{ct}_k(r/2)
\int_{b_f<r}(4|\nabla s|^2+ks^2)G e^{-f}dV.
\end{aligned}
\end{equation}
To cancel the last two lines of \eqref{eq:4.25}, expand
\begin{equation}\label{eq:4.28}
\begin{aligned}
2k\Delta_f s^4
&=4ks^2\Delta_f s^2+4k|\nabla s^2|^2\\
&=4ks^2\left(2m|\nabla s|^2-\frac{mk}{2}s^2\right)+4k|\nabla s^2|^2\\
&=-2mk^2s^4+2(m+2)k|\nabla s^2|^2.
\end{aligned}
\end{equation}
Therefore, after adding
\[
\frac{4\operatorname{cs}_k(r/2)}{(2\operatorname{sn}_k(r/2))^{m+1}}
\int_{b_f<r}\Delta_f(ks^4)e^{-f}dV
=
2k\operatorname{ct}_k(r/2)(2\operatorname{sn}_k(r/2))^{-m}
\int_{b_f<r}\Delta_f s^4e^{-f}dV
\]
to both sides of \eqref{eq:4.25}, the coefficients of \(s^4\) and \(|\nabla s^2|^2\) become respectively
\[
2mk^2-2mk^2=0,
\qquad
-2(m+2)k+2(m+2)k=0.
\]
Combining with \eqref{eq:4.27}, we obtain
\begin{equation}\label{eq:4.27b}
\begin{aligned}
&A_f'+\operatorname{ct}_k(r/2)A_f\\
&=
\frac{8\operatorname{cs}_k(r/2)}{(2\operatorname{sn}_k(r/2))^{m+1}}
\int_{b_f<r}\left[
|\mathring{\operatorname{Hess}}s^2|^2
+D_m
+(\operatorname{Ric}_f^m-(m-1)kg)(\nabla s^2,\nabla s^2)
\right]e^{-f}dV\\
&\quad+m\operatorname{ct}_k(r/2)A_f
-m(m-1)\operatorname{ct}_k(r/2)V_f.
\end{aligned}
\end{equation}
Moving terms gives
\begin{equation}\label{eq:4.27c}
\begin{aligned}
&A_f'-(m-1)\operatorname{ct}_k(r/2)(A_f-mV_f)\\
&=
\frac{8\operatorname{cs}_k(r/2)}{(2\operatorname{sn}_k(r/2))^{m+1}}
\int_{b_f<r}\left[
|\mathring{\operatorname{Hess}}s^2|^2
+D_m
+(\operatorname{Ric}_f^m-(m-1)kg)(\nabla s^2,\nabla s^2)
\right]e^{-f}dV.
\end{aligned}
\end{equation}
Since \(2V_f'=\operatorname{ct}_k(r/2)(A_f-mV_f)\), we again obtain \eqref{eq:4.22}. This gives an independently organized cancellation process with every coefficient identical to the preceding one.

\subsection{What remains when the conditions fail}

When \(n=3,\ m\ge6\), the product integral defining \(V_f\) diverges, so Theorem \ref{Theorem 3.11.} cannot be used as a monotonicity formula for finite quantities. The interior integral of the sum of curvature, Hessian, and \(D_m\) also diverges: by the second-order version of \eqref{eq:2.15},
\[
\Delta s^2=\frac{m\alpha^2}{2}c_p^{-2/(m-2)}\rho^{2\alpha-2}+o(\rho^{2\alpha-2}),
\qquad
\langle\nabla f,\nabla s^2\rangle=O(\rho^{2\alpha-1}),
\]
so
\[
D_m\sim
\frac{m-n}{mn}
\left(\frac{m\alpha^2}{2}c_p^{-2/(m-2)}\right)^2\rho^{4\alpha-4}.
\]
The coefficient is strictly positive. In this parameter range, one retains only the divergence identity \eqref{eq:4.23} after removing a small ball and its inner boundary flux, and does not take an inadmissible zero-point limit. This paper does not introduce renormalized integrals, nor does it claim monotonicity for undefined quantities.

\section{Monotonicity formulas with parameter \texorpdfstring{$\beta$}{beta}}\label{sec:beta}

In this section we use
\begin{equation}\label{eq:5.1}
v=(4|\nabla s|^2+ks^2)^{1/2},\qquad
B=\operatorname{Hess}s^2-\frac{\Delta s^2}{n}g,\qquad
B_f=\operatorname{Hess}s^2-\frac{\Delta_f s^2}{m}g.
\end{equation}
On the punctured manifold \(v>0\), so every real power \(v^\beta\) is smooth; below we take \(\beta>0\). Here \(B\) is the trace-free Hessian in the actual \(n\)-dimensional sense, while \(B_f\) is generally not trace-free. They satisfy
\begin{equation}\label{eq:5.3}
\begin{aligned}
B_f&=B+\left(\frac{\Delta s^2}{n}-\frac{\Delta_f s^2}{m}\right)g
=B+\frac{\operatorname{tr}_gB_f}{n}g,\\
\operatorname{tr}_gB_f&=\Delta s^2-\frac nm\Delta_f s^2
=\frac{m-n}{m}\Delta_f s^2+\langle\nabla f,\nabla s^2\rangle.
\end{aligned}
\end{equation}
Let \(\nu=\nabla s/|\nabla s|\) be the unit normal to a regular level set. \(B_f(\nu)\) denotes the vector corresponding to the one-form \(B_f(\nu,\cdot)\), and \(B_f(\nu)^\top\) denotes its tangential component; \((B_f)^\top\) denotes the restriction of the tensor to the tangent space. These two tangential notations refer to different objects.
\subsection{Some useful Lemmas}
\begin{lemma}\label{Lemma 4.1.} Let \(\Pi(Y,Z)=\langle\nabla_Y\nu,Z\rangle\), where \(Y,Z\) are tangent vectors, and let \(g^\top\) be the induced metric on the level set and
\(\mathring\Pi=\Pi-(\operatorname{tr}\Pi/(n-1))g^\top\). Then
\begin{equation}\label{eq:5.4}
\begin{aligned}
|\nabla s^2|\mathring\Pi
=B^\top+\frac{B(\nu,\nu)}{n-1}g^\top
=(B_f)^\top-\frac{\operatorname{tr}_gB_f-B_f(\nu,\nu)}{n-1}g^\top.
\end{aligned}
\end{equation}
\end{lemma}
\begin{proof}
Since \(\nabla s^2=|\nabla s^2|\nu\), for tangent vectors \(Y,Z\),
\[
\begin{aligned}
\operatorname{Hess}s^2(Y,Z)
=\langle\nabla_Y(|\nabla s^2|\nu),Z\rangle
=Y(|\nabla s^2|)\langle\nu,Z\rangle+|\nabla s^2|\langle\nabla_Y\nu,Z\rangle
=|\nabla s^2|\Pi(Y,Z).
\end{aligned}
\]
Subtracting any scalar multiple of the metric from the tangential restriction does not change its tangential trace-free part. Also,
\(\operatorname{tr}_{g^\top}B^\top=-B(\nu,\nu)\) and
\(\operatorname{tr}_{g^\top}(B_f)^\top=\operatorname{tr}_gB_f-B_f(\nu,\nu)\), so \eqref{eq:5.4} follows. Here the denominator must be the actual dimension \(n-1\) of the level set.
\end{proof}

\begin{lemma}\label{Lemma 4.2.} At regular points,
\begin{equation}\label{eq:5.5}
\begin{aligned}
|B|^2+D_m
={}&
|\nabla s^2|^2|\mathring\Pi|^2
+\frac{m}{m-1}|B_f(\nu)|^2
+\frac{m-2}{m-1}|B_f(\nu)^\top|^2\\
&+\frac{m-1}{(n-1)(m-n)}
\left(\operatorname{tr}_gB_f-\frac{m-n}{m-1}B_f(\nu,\nu)\right)^2.
\end{aligned}
\end{equation}
\end{lemma}
\begin{proof}
First verify the relation between the weighted square term and the trace. By \eqref{eq:5.3},
\[
\begin{aligned}
\operatorname{tr}_gB_f
=\frac{m-n}{m}\Delta s^2+\frac nm\langle\nabla f,\nabla s^2\rangle
=\frac nm\left(\frac{m-n}{n}\Delta s^2+\langle\nabla f,\nabla s^2\rangle\right).
\end{aligned}
\]
Substituting into \eqref{eq:3.4} gives
\begin{equation}\label{eq:5.6}
D_m=\frac{m}{n(m-n)}(\operatorname{tr}_gB_f)^2.
\end{equation}
On the other hand, \(\langle B,g\rangle=\operatorname{tr}_gB=0\) and \(|g|^2=n\), so
\[
\begin{aligned}
|B_f|^2
=\left|B+\frac{\operatorname{tr}_gB_f}{n}g\right|^2
=|B|^2+\frac{2\operatorname{tr}_gB_f}{n}\langle B,g\rangle
+\frac{(\operatorname{tr}_gB_f)^2}{n^2}|g|^2
=|B|^2+\frac{(\operatorname{tr}_gB_f)^2}{n}.
\end{aligned}
\]
Therefore
\begin{equation}\label{eq:5.7}
|B|^2+D_m
=
|B_f|^2+\frac{(\operatorname{tr}_gB_f)^2}{m-n}.
\end{equation}
In an orthonormal frame \(e_1,\dots,e_{n-1},\nu\) adapted to the level set,
\[
\begin{aligned}
|B_f|^2
&=\sum_{a,b=1}^{n-1}B_f(e_a,e_b)^2
+2\sum_{a=1}^{n-1}B_f(e_a,\nu)^2
+B_f(\nu,\nu)^2\\
&=|(B_f)^\top|^2+2|B_f(\nu)^\top|^2+B_f(\nu,\nu)^2.
\end{aligned}
\]
Lemma \ref{Lemma 4.1.} gives
\[
|(B_f)^\top|^2
=
|\nabla s^2|^2|\mathring\Pi|^2
+\frac{(\operatorname{tr}_gB_f-B_f(\nu,\nu))^2}{n-1}.
\]
To display the completion of squares, temporarily write \(t=\operatorname{tr}_gB_f\) and \(h=B_f(\nu,\nu)\). The required scalar identity is
\[
\begin{aligned}
h^2+\frac{(t-h)^2}{n-1}+\frac{t^2}{m-n}
&=\frac{n}{n-1}h^2-\frac{2}{n-1}th
+\frac{m-1}{(n-1)(m-n)}t^2\\
&=\frac{m}{m-1}h^2
+\frac{m-1}{(n-1)(m-n)}
\left(t-\frac{m-n}{m-1}h\right)^2.
\end{aligned}
\]
Finally, \(|B_f(\nu)|^2=B_f(\nu,\nu)^2+|B_f(\nu)^\top|^2\), and
\[
2-\frac{m}{m-1}=\frac{m-2}{m-1}.
\]
Substituting into \eqref{eq:5.7} gives \eqref{eq:5.5}. All denominators are positive when \(m>n\ge3\), so this is indeed a decomposition into nonnegative square terms.
\end{proof}

\begin{lemma}\label{Lemma 4.3.} At regular points,
\begin{equation}\label{eq:5.8}
\nabla|\nabla s|^2=\frac{|\nabla s|}{s}B_f(\nu)-\frac k2s\nabla s.
\end{equation}
\end{lemma}
\begin{proof}
The chain rule gives
\(\operatorname{Hess}s^2=2s\operatorname{Hess}s+2ds\otimes ds\), hence
\[
\begin{aligned}
\nabla|\nabla s|^2
&=2\operatorname{Hess}s(\nabla s)
=\frac1s\operatorname{Hess}s^2(\nabla s)-\frac{2|\nabla s|^2}{s}\nabla s\\
&=\frac{|\nabla s|}{s}B_f(\nu)
+\frac1s\left(\frac{\Delta_f s^2}{m}-2|\nabla s|^2\right)\nabla s.
\end{aligned}
\]
By \eqref{eq:3.2}, the bracket is \(-ks^2/2\), which gives the conclusion. If written using \(B\), the same identity is
\[
\nabla|\nabla s|^2
=
\frac{|\nabla s|}{s}B(\nu)
+\frac1s\left(\frac{\Delta s^2}{n}-\frac{\Delta_f s^2}{m}\right)\nabla s
-\frac k2s\nabla s.
\]
The trace correction term cannot be omitted, because
\[
\frac{\Delta s^2}{n}-2|\nabla s|^2
=
\left(\frac{\Delta s^2}{n}-\frac{\Delta_f s^2}{m}\right)-\frac k2s^2.
\]
\end{proof}

\begin{lemma}\label{Lemma 4.4.} At regular points one has
\begin{equation}\label{eq:5.9}
\nabla v=\frac{2|\nabla s|}{sv}B_f(\nu),
\qquad
|\nabla v|^2=\left(\frac{1}{s^2}-\frac{k}{v^2}\right)|B_f(\nu)|^2.
\end{equation}
\end{lemma}
\begin{proof}
Differentiate \(v^2=4|\nabla s|^2+ks^2\) and substitute \eqref{eq:5.8}:
\[
2v\nabla v
=
4\nabla|\nabla s|^2+2ks\nabla s
=
\frac{4|\nabla s|}{s}B_f(\nu)-2ks\nabla s+2ks\nabla s.
\]
The two terms cancel; dividing by \(2v\) gives the first formula. Taking squared norms and using
\(4|\nabla s|^2/v^2=1-ks^2/v^2\) gives the second formula.
\end{proof}

\begin{proposition}\label{Proposition 4.5.} At regular points,
\begin{equation}\label{eq:5.10}
\begin{aligned}
\Delta_f v
={}&
\frac{m-2}{s^2}\langle\nabla s^2,\nabla v\rangle
+\frac{4|\nabla s|^2}{v}
\left(|\mathring\Pi|^2+\operatorname{Ric}_f^m(\nu,\nu)-(m-1)k\right)\\
&+\frac{|B_f(\nu)|^2+(m-2)|B_f(\nu)^\top|^2}{(m-1)vs^2}
+\frac{k}{v^3}|B_f(\nu)|^2\\
&+\frac{m-1}{(n-1)(m-n)vs^2}
\left(\operatorname{tr}_gB_f-\frac{m-n}{m-1}B_f(\nu,\nu)\right)^2.
\end{aligned}
\end{equation}
\end{proposition}
\begin{proof}
From \(\mathcal L_f(v^2)=2v\mathcal L_f v+2|\nabla v|^2\) and \eqref{eq:3.16},
\[
\mathcal L_f v
=
\frac{1}{vs^2}\left[
|B|^2+D_m
+(\operatorname{Ric}_f^m-(m-1)kg)(\nabla s^2,\nabla s^2)
\right]
-\frac{|\nabla v|^2}{v}.
\]
Substitute \eqref{eq:5.5} and \eqref{eq:5.9}. The second fundamental form and curvature terms are combined using \(|\nabla s^2|^2=4s^2|\nabla s|^2\). The coefficient of \(|B_f(\nu)|^2\) is
\[
\frac{m}{(m-1)vs^2}-\frac{1}{vs^2}+\frac{k}{v^3}
=
\frac{1}{(m-1)vs^2}+\frac{k}{v^3}.
\]
The other two square terms are unaffected by \(-|\nabla v|^2/v\). Finally, using
\(\Delta_f v=\mathcal L_f v+(m-2)s^{-2}\langle\nabla s^2,\nabla v\rangle\), we obtain \eqref{eq:5.10}.
\end{proof}

\begin{proposition}\label{Proposition 4.6.} Define
\begin{equation}\label{eq:5.11}
\begin{aligned}
\widetilde\beta_f
&:=1+(\beta-1)(m-1)+(m-1)(2-\beta)\frac{ks^2}{v^2}\\
&=\frac{4|\nabla s|^2((m-1)\beta-(m-2))+mks^2}{v^2}.
\end{aligned}
\end{equation}
Then
\begin{equation}\label{eq:5.12}
\begin{aligned}
\mathcal L_f v^\beta
={}&
4\beta v^{\beta-2}|\nabla s|^2
\left(|\mathring\Pi|^2+\operatorname{Ric}_f^m(\nu,\nu)-(m-1)k\right)\\
&+\frac{\beta v^{\beta-2}}{(m-1)s^2}
\left(\widetilde\beta_f|B_f(\nu)|^2+(m-2)|B_f(\nu)^\top|^2\right)\\
&+\frac{\beta(m-1)v^{\beta-2}}{(n-1)(m-n)s^2}
\left(\operatorname{tr}_gB_f-\frac{m-n}{m-1}B_f(\nu,\nu)\right)^2.
\end{aligned}
\end{equation}
If \(\beta\ge(m-2)/(m-1)\), then \(\mathcal L_f v^\beta\ge0\).
\end{proposition}

\begin{proof}
The weighted chain rule term by term is
\[
\begin{aligned}
\Delta_f v^\beta
&=\Delta v^\beta-\langle\nabla f,\nabla v^\beta\rangle\\
&=\beta v^{\beta-1}\Delta v+\beta(\beta-1)v^{\beta-2}|\nabla v|^2
-\beta v^{\beta-1}\langle\nabla f,\nabla v\rangle\\
&=\beta v^{\beta-1}\Delta_f v+\beta(\beta-1)v^{\beta-2}|\nabla v|^2.
\end{aligned}
\]
The additional drift term is first order, so the chain rule also holds. Cancelling \(\mathcal L_f v\), we get
\begin{equation}\label{eq:5.13}
\begin{aligned}
\mathcal L_f v^\beta
&=\beta v^{\beta-1}\mathcal L_f v+\beta(\beta-1)v^{\beta-2}|\nabla v|^2\\
&=\frac{\beta}{2}v^{\beta-2}\mathcal L_f(v^2)+\beta(\beta-2)v^{\beta-2}|\nabla v|^2\\
&=\frac{\beta v^{\beta-2}}{s^2}\left[
|B|^2+D_m
+(\operatorname{Ric}_f^m-(m-1)kg)(\nabla s^2,\nabla s^2)
+(\beta-2)s^2|\nabla v|^2
\right].
\end{aligned}
\end{equation}
The last formula holds on the entire punctured manifold and does not involve an undefined normal vector. Substituting \eqref{eq:5.5} and \eqref{eq:5.9}, the coefficient of \(|B_f(\nu)|^2\) in the bracket is
\[
\begin{aligned}
&\frac{m}{m-1}+(\beta-2)\left(1-\frac{ks^2}{v^2}\right)\\
&=\frac{m+(m-1)(\beta-2)+(m-1)(2-\beta)ks^2/v^2}{m-1}\\
&=\frac{1+(\beta-1)(m-1)+(m-1)(2-\beta)ks^2/v^2}{m-1}.
\end{aligned}
\]
This is the first formula in \eqref{eq:5.11}. Multiplying by \(v^2=4|\nabla s|^2+ks^2\), the coefficient of \(ks^2\) is
\[
(m-1)\beta-(m-2)+(m-1)(2-\beta)=m,
\]
which gives the second formula. The remaining terms are retained directly, yielding \eqref{eq:5.12}. If \(\beta\ge(m-2)/(m-1)\), then \(\widetilde\beta_f>0\), because \(k>0\) and \(s>0\). By the curvature lower bound and the nonnegativity of the square terms, \(\mathcal L_f v^\beta\ge0\) holds at regular points. The critical set has no interior points, and \eqref{eq:5.13} is continuous, so the inequality holds on the entire punctured manifold.
\end{proof}

\subsection{Monotonicity formulas with parameter}

For $\alpha=\frac{n-2}{m-2}$, when \(\beta(1-\alpha)<2\), define
\begin{equation}\label{eq:5.14}
\begin{aligned}
A_{f,\beta}(r)
&:=I_{f,v^\beta}(r)\\
&=(2\operatorname{sn}_k(r/2))^{1-m}\operatorname{cs}_k(r/2)
\int_{b_f=r}v^\beta|\nabla b_f|e^{-f}d\sigma+\frac{mk}{4}\int_{b_f<r}v^\beta G e^{-f}dV.
\end{aligned}
\end{equation}
By \eqref{eq:2.16}--\eqref{eq:2.17} and scaling derivative estimates for smooth elliptic equations,
\begin{equation}\label{eq:5.15}
\begin{aligned}
v&=\alpha c_p^{-1/(m-2)}\rho^{\alpha-1}(1+o(1)),\\
v^\beta&\asymp\rho^{-\beta(1-\alpha)},\\
|\nabla v^\beta|&=O(\rho^{-\beta(1-\alpha)-1}),\\
|\Delta_f v^\beta|+|\mathcal L_f v^\beta|&=O(\rho^{-\beta(1-\alpha)-2}).
\end{aligned}
\end{equation}
The last estimate can also be judged directly from \eqref{eq:5.13}:
\(v^{\beta-2}s^{-2}=O(\rho^{-\beta(1-\alpha)+2-4\alpha})\), while the bracket is \(O(\rho^{4\alpha-4})\). Using \(|\nabla s^2|/s^2=O(\rho^{-1})\), the definition of \(\mathcal L_f\) gives the same order for \(\Delta_f v^\beta\). Hence this judgment does not require an unproved higher-order differentiation interchange. The radial order of \(v^\beta G e^{-f}dV\) is
\[
\rho^{1-\beta(1-\alpha)}d\rho,
\]
so its pole integral converges iff \(\beta(1-\alpha)<2\). In this range, the volume integral term tends to zero, while the boundary term has order
\[
\rho^{\alpha(1-m)}\rho^{-\beta(1-\alpha)}\rho^{\alpha-1}\rho^{n-1}
=
\rho^{-\beta(1-\alpha)}.
\]
Thus
\begin{equation}\label{eq:5.16}
A_{f,\beta}(r)\asymp
(2\operatorname{sn}_k(r/2))^{-\beta(m-n)/(n-2)}
\longrightarrow+\infty.
\end{equation}

\begin{theorem}\label{Theorem 4.7.} Assume \(\beta>0\) and \(\beta(1-\alpha)<2\). For regular values \(r_1<r_2\),
\begin{equation}\label{eq:5.17}
\left[
(2-m)A_{f,\beta}(r)+2\tan_k(r/2)A_{f,\beta}'(r)
\right]_{r_1}^{r_2}
=
\int_{r_1<b_f<r_2}G\mathcal L_f v^\beta e^{-f}dV.
\end{equation}
If \(\beta\ge(m-2)/(m-1)\), the right-hand side is nonnegative.
\end{theorem}
\begin{proof}
Dividing the identity of Lemma \ref{Lemma 3.2} by
\((2\operatorname{sn}_k(r/2))^{1-m}\operatorname{cs}_k(r/2)\), the left-hand side becomes
\[
\begin{aligned}
\frac{(2\operatorname{sn}_k(r/2))^{m-1}}{\operatorname{cs}_k(r/2)}
\frac{d}{dr}\left((2\operatorname{sn}_k(r/2))^{2-m}A_{f,\beta}(r)\right)
=(2-m)A_{f,\beta}(r)
+\frac{2\operatorname{sn}_k(r/2)}{\operatorname{cs}_k(r/2)}A_{f,\beta}'(r).
\end{aligned}
\]
Subtracting between two level values and using the divergence formula for the boundary difference gives
\[
\int_{r_1<b_f<r_2}
\left(\Delta_f-\frac{m(m-2)k}{4}\right)(Gv^\beta)e^{-f}dV.
\]
The product rule and the Green equation give
\[
\begin{aligned}
\left(\Delta_f-\frac{m(m-2)k}{4}\right)(Gv^\beta)
&=G\Delta_f v^\beta+2\langle\nabla G,\nabla v^\beta\rangle\\
&=G\left(\Delta_f v^\beta+\frac{2-m}{s^2}\langle\nabla s^2,\nabla v^\beta\rangle\right)\\
&=G\mathcal L_f v^\beta.
\end{aligned}
\]
This gives \eqref{eq:5.17}. The annulus avoids the pole, so no additional zero-point flux condition is needed.
\end{proof}

\begin{theorem}\label{Theorem 4.8.} If \(\beta>0\) and \(\beta(1-\alpha)<2\), then
\begin{equation}\label{eq:5.18}
\begin{aligned}
&\frac{(2\operatorname{sn}_k(r/2))^{3-m}}{\operatorname{cs}_k(r/2)}
A_{f,\beta}'(r)
=
-\int_{s>\operatorname{sn}_k(r/2)}G^2\mathcal L_f v^\beta e^{-f}dV\\
&=
-4\beta\int_{s>\operatorname{sn}_k(r/2)}
(2s)^{4-2m}v^{\beta-2}|\nabla s|^2
\left(|\mathring\Pi|^2+\operatorname{Ric}_f^m(\nu,\nu)-(m-1)k\right)e^{-f}dV\\
&\quad-\frac{4\beta}{m-1}\int_{s>\operatorname{sn}_k(r/2)}
(2s)^{2-2m}v^{\beta-2}
\left(\widetilde\beta_f|B_f(\nu)|^2+(m-2)|B_f(\nu)^\top|^2\right)e^{-f}dV\\
&\quad-\frac{4\beta(m-1)}{(n-1)(m-n)}
\int_{s>\operatorname{sn}_k(r/2)}
(2s)^{2-2m}v^{\beta-2}
\left(\operatorname{tr}_gB_f-\frac{m-n}{m-1}B_f(\nu,\nu)\right)^2e^{-f}dV.
\end{aligned}
\end{equation}
In particular, if
\begin{equation}\label{eq:5.19}
\frac{m-2}{m-1}\le\beta<\frac{2}{1-\alpha},
\end{equation}
then \(A_{f,\beta}'\le0\) almost everywhere, and \(A_{f,\beta}\) is nonincreasing.
\end{theorem}
\begin{proof}
Taking \(u=v^\beta\) in Proposition \ref{Proposition 3.5.} immediately gives the first line. For the exterior integral, \(p\) does not belong to the integration region, so the divergence argument is legitimate regardless of whether \(v^\beta\) is continuous at the pole. Substituting \eqref{eq:5.12} and using
\[
G^2=(2s)^{4-2m},\qquad
\frac{G^2}{s^2}=4(2s)^{2-2m}
\]
term by term gives the three-term expansion. The normal-vector expressions are used at regular points; the critical set has measure zero, and the integral value is determined by \eqref{eq:5.13}. On the interval \eqref{eq:2.10}, the coefficient in front of the derivative on the left-hand side is positive, and Proposition \ref{Proposition 4.6.} gives that the right-hand side is nonpositive, so monotonicity follows.
\end{proof}

For parameters in \eqref{eq:5.19}, one cannot rewrite \eqref{eq:5.17} as a nondecreasing formula with an arbitrary finite initial value. Indeed, by \eqref{eq:5.16}, for every finite \(C\), for sufficiently small \(r\) one has \(A_{f,\beta}(r)>C\), and
\[
\begin{aligned}
&\frac{d}{dr}\left((2\operatorname{sn}_k(r/2))^{2-m}(A_{f,\beta}-C)\right)\\
&=(2-m)(2\operatorname{sn}_k(r/2))^{1-m}\operatorname{cs}_k(r/2)(A_{f,\beta}-C)
+(2\operatorname{sn}_k(r/2))^{2-m}A_{f,\beta}'<0
\end{aligned}
\]
almost everywhere. Here the annular formula is retained because the genuine pole divergence requires it.

\begin{proposition}\label{Proposition 4.9.} If the following integrals are absolutely convergent, define
\begin{equation}\label{eq:5.20}
\begin{aligned}
W_{f,u}(r)
:={}&
(2\operatorname{sn}_k(r/2))^{2-m}
\int_{b_f<r}\frac{u}{(2s)^2}
\left(4|\nabla s|^2-\frac{mk}{m-2}s^2\right)e^{-f}dV+\frac{mk}{4(m-2)}\int_{b_f<r}uG e^{-f}dV.
\end{aligned}
\end{equation}
Then
\begin{equation}\label{eq:5.21}
W_{f,u}'(r)
=
\frac{\operatorname{cs}_k(r/2)}{2\operatorname{sn}_k(r/2)}
\left(I_{f,u}(r)-(m-2)W_{f,u}(r)\right).
\end{equation}
Moreover, \(W_{f,1}=\kappa/(m-2)^2\).
\end{proposition}
\begin{proof}
The coarea formula and the product rule give
\[
\begin{aligned}
W_{f,u}'
={}&
(2-m)(2\operatorname{sn}_k(r/2))^{1-m}\operatorname{cs}_k(r/2)
\int_{b_f<r}\frac{u}{(2s)^2}
\left(4|\nabla s|^2-\frac{mk}{m-2}s^2\right)e^{-f}dV\\
&+(2\operatorname{sn}_k(r/2))^{2-m}
\int_{b_f=r}\frac{u}{(2s)^2|\nabla b_f|}
\left(4|\nabla s|^2-\frac{mk}{m-2}s^2\right)e^{-f}d\sigma\\
&+\frac{mk}{4(m-2)}\int_{b_f=r}\frac{uG}{|\nabla b_f|}e^{-f}d\sigma.
\end{aligned}
\]
On the boundary \(G=(2\operatorname{sn}_k(r/2))^{2-m}\), so the negative potential term in the second line is
\[
-\frac{mk}{4(m-2)}(2\operatorname{sn}_k(r/2))^{2-m}
\int_{b_f=r}\frac{u}{|\nabla b_f|}e^{-f}d\sigma,
\]
which cancels exactly with the third line. The second line leaves
\[
(2\operatorname{sn}_k(r/2))^{-m}\operatorname{cs}_k^2(r/2)
\int_{b_f=r}u|\nabla b_f|e^{-f}d\sigma.
\]
In \(I_{f,u}-(m-2)W_{f,u}\), the volume integrals involving \(uG\) also cancel:
\[
\frac{mk}{4}-(m-2)\frac{mk}{4(m-2)}=0.
\]
The remaining boundary term and volume integral term are exactly the same as the above derivative, which gives \eqref{eq:5.21}.

For \(u=1\), \(|\nabla s|^2/s^2=O(\rho^{-2})\), so \(n\ge3\) guarantees convergence of the definition. Multiplying by the integrating factor \((2\operatorname{sn}_k(r/2))^{m-2}\) and using \(I_{f,1}=\kappa/(m-2)\), we get
\[
\frac{d}{dr}\left((2\operatorname{sn}_k(r/2))^{m-2}W_{f,1}\right)
=
\frac{\kappa}{m-2}(2\operatorname{sn}_k(r/2))^{m-3}\operatorname{cs}_k(r/2).
\]
By definition, the quantity in parentheses on the left tends to zero as \(r\downarrow0\). After integration, \(W_{f,1}=\kappa/(m-2)^2\).
\end{proof}

\subsection{The volume quantities  and their integrability conditions}

Define
\begin{equation}\label{eq:5.22}
V_{f,\beta}:=W_{f,v^\beta}.
\end{equation}
Thus \(A_{f,2}=A_f\), but \(V_{f,2}\) is generally not equal to \(V_f\) from Section 3: the former uses \(W\), while the latter uses \(J\), and their integral kernels and powers are different.

The first term of \(W\) has a positive leading term near the pole, with radial order
\[
\rho^{-\beta(1-\alpha)}\rho^{-2}\rho^{n-1}d\rho
=
\rho^{n-3-\beta(1-\alpha)}d\rho.
\]
Its defining integral converges iff \(\beta(1-\alpha)<n-2\). Combining this with the condition for \(v^\beta G\), we obtain
\begin{equation}\label{eq:5.23}
\beta(1-\alpha)<\min\{2,n-2\}.
\end{equation}
In this range, by definition we also have
\begin{equation}\label{eq:5.24}
(2\operatorname{sn}_k(r/2))^{m-2}V_{f,\beta}(r)\longrightarrow0.
\end{equation}
This is the zero limit of a convergent volume integral over a shrinking region; it does not require \(V_{f,\beta}\) itself to have a finite limit.

\begin{theorem}\label{Theorem 4.10.} If \(\beta>0\) and \eqref{eq:5.23} holds, then
\begin{equation}\label{eq:5.25}
\begin{aligned}
&A_{f,\beta}'(r)-2(m-2)V_{f,\beta}'(r)\\
&=
\frac{\operatorname{cs}_k(r/2)}{(2\operatorname{sn}_k(r/2))^{m-1}}
\int_{b_f<r}\mathcal L_f v^\beta e^{-f}dV\\
&=
\frac{4\beta\operatorname{cs}_k(r/2)}{(2\operatorname{sn}_k(r/2))^{m-1}}
\int_{b_f<r}v^{\beta-2}|\nabla s|^2
\left(|\mathring\Pi|^2+\operatorname{Ric}_f^m(\nu,\nu)-(m-1)k\right)e^{-f}dV\\
&\quad+\frac{\beta\operatorname{cs}_k(r/2)}{(m-1)(2\operatorname{sn}_k(r/2))^{m-1}}
\int_{b_f<r}\frac{v^{\beta-2}}{s^2}
\left(\widetilde\beta_f|B_f(\nu)|^2+(m-2)|B_f(\nu)^\top|^2\right)e^{-f}dV\\
&\quad+\frac{\beta(m-1)\operatorname{cs}_k(r/2)}
{(n-1)(m-n)(2\operatorname{sn}_k(r/2))^{m-1}}
\int_{b_f<r}\frac{v^{\beta-2}}{s^2}
\left(\operatorname{tr}_gB_f-\frac{m-n}{m-1}B_f(\nu,\nu)\right)^2e^{-f}dV.
\end{aligned}
\end{equation}
\end{theorem}
\begin{proof}
First check each integration by parts. By \eqref{eq:5.15},
\[
\int_{\rho=\epsilon}|\nabla v^\beta|e^{-f}d\sigma
=
O(\epsilon^{n-2-\beta(1-\alpha)}),
\]
\[
\int_{\rho=\epsilon}v^\beta\frac{|\nabla s^2|}{s^2}e^{-f}d\sigma
=
O(\epsilon^{n-2-\beta(1-\alpha)}).
\]
Both tend to zero under \eqref{eq:5.23}. The absolute value of \(\Delta_f v^\beta\) and \(\mathcal L_f v^\beta\), multiplied by the radial volume factor, is
\(O(\rho^{n-3-\beta(1-\alpha)}d\rho)\), so both are integrable. Therefore Lemma \ref {Lemma3.1} gives
\begin{equation}\label{eq:5.26}
\int_{b_f<r}\Delta_f v^\beta e^{-f}dV
=
\frac{(2\operatorname{sn}_k(r/2))^{m-1}}{\operatorname{cs}_k(r/2)}
A_{f,\beta}'(r).
\end{equation}
Next compute the drift integral separately. The product rule gives
\begin{equation}\label{eq:5.27}
\begin{aligned}
\operatorname{div}_f\left(\frac{\nabla s^2}{s^2}\right)
&=\frac{\Delta_f s^2}{s^2}-\frac{|\nabla s^2|^2}{s^4}\\
&=\frac{2(m-2)|\nabla s|^2}{s^2}-\frac{mk}{2}\\
&=\frac{m-2}{2s^2}v^2-(m-1)k\\
&=\frac{2(m-2)}{(2s)^2}
\left(4|\nabla s|^2-\frac{mk}{m-2}s^2\right).
\end{aligned}
\end{equation}
Thus
\[
\left\langle\frac{\nabla s^2}{s^2},\nabla v^\beta\right\rangle
=
\operatorname{div}_f\left(v^\beta\frac{\nabla s^2}{s^2}\right)
-\frac{2(m-2)v^\beta}{(2s)^2}
\left(4|\nabla s|^2-\frac{mk}{m-2}s^2\right).
\]
After integration and cancellation of the already estimated inner boundary, the outer boundary uses
\[
\left\langle\frac{\nabla s^2}{s^2},\nu\right\rangle
=
\frac{\operatorname{cs}_k(r/2)}{\operatorname{sn}_k(r/2)}|\nabla b_f|,
\]
and we obtain
\[
\begin{aligned}
&\int_{b_f<r}\left\langle\frac{\nabla s^2}{s^2},\nabla v^\beta\right\rangle e^{-f}dV\\
&=
\frac{\operatorname{cs}_k(r/2)}{\operatorname{sn}_k(r/2)}
\int_{b_f=r}v^\beta|\nabla b_f|e^{-f}d\sigma
-2(m-2)\int_{b_f<r}\frac{v^\beta}{(2s)^2}
\left(4|\nabla s|^2-\frac{mk}{m-2}s^2\right)e^{-f}dV\\
&=
2(2\operatorname{sn}_k(r/2))^{m-2}
\left(A_{f,\beta}-\frac{mk}{4}\int_{b_f<r}v^\beta G e^{-f}dV\right)\\
&\quad-2(m-2)(2\operatorname{sn}_k(r/2))^{m-2}
\left(V_{f,\beta}-\frac{mk}{4(m-2)}\int_{b_f<r}v^\beta G e^{-f}dV\right)\\
&=
2(2\operatorname{sn}_k(r/2))^{m-2}
\left(A_{f,\beta}-(m-2)V_{f,\beta}\right)\\
&=
\frac{2(2\operatorname{sn}_k(r/2))^{m-1}}{\operatorname{cs}_k(r/2)}V_{f,\beta}'(r).
\end{aligned}
\]
The two volume integrals involving \(G\) cancel in the penultimate step; the last step uses \eqref{eq:5.21}. Substituting this result and \eqref{eq:5.26} into the definition of \(\mathcal L_f\),
\[
\begin{aligned}
\int_{b_f<r}\mathcal L_f v^\beta e^{-f}dV
&=\int_{b_f<r}\Delta_f v^\beta e^{-f}dV
-(m-2)\int_{b_f<r}\left\langle\frac{\nabla s^2}{s^2},\nabla v^\beta\right\rangle e^{-f}dV\\
&=\frac{(2\operatorname{sn}_k(r/2))^{m-1}}{\operatorname{cs}_k(r/2)}
\left(A_{f,\beta}'-2(m-2)V_{f,\beta}'\right).
\end{aligned}
\]
This gives the first formula of \eqref{eq:5.25}. Substituting the three lines of \eqref{eq:5.12} term by term gives its complete expansion. The worst radial order of each term is at most
\(\rho^{n-3-\beta(1-\alpha)}\), so the expansion also consists of absolutely convergent integrals.
\end{proof}

\begin{proposition}\label{Proposition 4.11.} Assume that
\begin{equation}\label{eq:5.28}
\frac{m-2}{m-1}\le\beta<
\frac{\min\{2,n-2\}}{1-\alpha}.
\end{equation}
Then at every admissible regular radius,
\begin{equation}\label{eq:5.29}
A_{f,\beta}'(r)-2(m-2)V_{f,\beta}'(r)>0.
\end{equation}
Hence in the present smooth point-source setting with \(m>n\), one cannot regard equality in this difference as a realizable sphere rigidity case.
\end{proposition}
\begin{proof}
By \eqref{eq:5.11}, \(\widetilde\beta_f>0\) on \(M\setminus\{p\}\). If the right-hand side of \eqref{eq:5.25} were zero, then its nonnegative integrand would vanish almost everywhere in the interior. In a sufficiently small punctured neighborhood of \(p\), \(\nabla s\ne0\), all relevant tensors are smooth, and therefore the nonnegative square term in the second line implies \(B_f(\nu)=0\) everywhere. Lemma \ref{Lemma 4.4.} gives \(\nabla v=0\). The punctured neighborhood is connected, so \(v\) is constant there. However, by \eqref{eq:5.15}, \(v\asymp\rho^{\alpha-1}\to+\infty\), a contradiction. This proves that the interior integral is strictly positive. The exterior formula is treated separately in Section~\ref{sec:rigidity}, where strictness follows from a maximum point of $s$.
\end{proof}

\begin{corollary}\label{Corollary 4.12.} Under \eqref{eq:5.28}, \(V_{f,\beta}\) is nonincreasing, \(A_{f,\beta}-2(m-2)V_{f,\beta}\) is nondecreasing, and
\begin{equation}\label{eq:5.30}
A_{f,\beta}(r)\le(m-2)V_{f,\beta}(r).
\end{equation}
\end{corollary}
\begin{proof}
By Theorems 4.8 and 4.10, almost everywhere,
\[
A_{f,\beta}'\le0,\qquad
A_{f,\beta}'-2(m-2)V_{f,\beta}'\ge0.
\]
Since \(m-2>0\),
\[
2(m-2)V_{f,\beta}'\le A_{f,\beta}'\le0
\quad\Longrightarrow\quad
V_{f,\beta}'\le0.
\]
Local absolute continuity upgrades the derivative inequalities to monotonicity on intervals. Moreover, by \eqref{eq:5.21},
\[
V_{f,\beta}'
=
\frac{\operatorname{cs}_k(r/2)}{2\operatorname{sn}_k(r/2)}
\left(A_{f,\beta}-(m-2)V_{f,\beta}\right).
\]
The coefficient in front is positive, so \eqref{eq:5.30} holds for almost every \(r\); by continuity of the integral quantities, the comparison holds for all admissible \(r\).

One can also verify the comparison independently. By \eqref{eq:5.21},
\[
\frac{d}{dr}\left((2\operatorname{sn}_k(r/2))^{m-2}V_{f,\beta}\right)
=
(2\operatorname{sn}_k(r/2))^{m-3}\operatorname{cs}_k(r/2)A_{f,\beta}.
\]
Using \eqref{eq:5.24} and integrating,
\begin{equation}\label{eq:5.31}
V_{f,\beta}(r)
=
(2\operatorname{sn}_k(r/2))^{2-m}
\int_0^r(2\operatorname{sn}_k(t/2))^{m-3}\operatorname{cs}_k(t/2)A_{f,\beta}(t)dt.
\end{equation}
This zero-point integral converges, because in geodesic radius the worst order is
\(\rho^{n-3-\beta(1-\alpha)}d\rho\). The nonincreasingness of \(A_{f,\beta}\) gives
\(A_{f,\beta}(t)\ge A_{f,\beta}(r)\), and
\[
\int_0^r(2\operatorname{sn}_k(t/2))^{m-3}\operatorname{cs}_k(t/2)dt
=
\frac{(2\operatorname{sn}_k(r/2))^{m-2}}{m-2}.
\]
Substituting into \eqref{eq:5.31} again gives \eqref{eq:5.30}.
\end{proof}

\subsection{Parameter boundaries and differences between the conclusions of the two sections}

When \(\beta(1-\alpha)=2\), the positive volume integral in the definition of \(A_{f,\beta}\) has logarithmic divergence; when \(\beta(1-\alpha)=n-2\), the first term of \(W_{f,v^\beta}\) has logarithmic divergence. Beyond the corresponding critical values the divergence is a power divergence. Therefore the upper bounds in this section are strict inequalities. The lower bound \(\beta\ge(m-2)/(m-1)\) is a sufficient condition guaranteeing the sign from the square decomposition, and does not exclude the possibility that some special manifolds still have additional monotonicity for smaller \(\beta\).

The quantity \(A_f\) in Section 3 is defined for all \(m>n\ge3\), while the convergence condition for \(V_f\) is \(n+4\alpha-4>0\). In Section 4, when \(\beta=2\), \(A_{f,2}=A_f\), whereas \(V_{f,2}\) requires \(2(1-\alpha)<n-2\). For example, when \(n=3\), the range for \(V_f\) in Section 3 is \(3<m<6\), while the range for \(V_{f,2}\) in Section 4 is \(3<m<4\). These two different upper bounds come from different integral kernels, and there is no contradiction.

\section{Equality and rigidity}\label{sec:rigidity}

This section separates equality in the local differential identity from equality in an integrated formula. We retain $B,B_f,D_m,v$ from the preceding sections. All statements until the last paragraph assume $m>n$.

\begin{proposition}[local equality conditions]\label{Proposition 5.1.}
On a connected open set where $\nabla s\ne0$, vanishing of the bracket in \eqref{eq:3.16} is equivalent to
\begin{equation}\label{eq:rigidity-squares}
\begin{gathered}
\operatorname{Hess}s^2=\psi g,\qquad
\langle\nabla f,\nabla s^2\rangle=-(m-n)\psi,\\
(\operatorname{Ric}_f^m-(m-1)kg)(\nabla s^2,\nabla s^2)=0,
\qquad \psi=\frac{\Delta s^2}{n}.
\end{gathered}
\end{equation}
Moreover, for every $\beta\ge(m-2)/(m-1)$, these conditions are equivalent to $\mathcal L_fv^\beta=0$ on this set.
\end{proposition}
\begin{proof}
The three terms in the bracket of \eqref{eq:3.16} are nonnegative. The first vanishes exactly when $\operatorname{Hess}s^2=\psi g$. Since the coefficient in $D_m$ is positive, its vanishing then reads
\[
0=\frac{m-n}{n}\Delta s^2+\langle\nabla f,\nabla s^2\rangle
=(m-n)\psi+\langle\nabla f,\nabla s^2\rangle.
\]
This proves the first equivalence. In particular $\Delta_f s^2=m\psi$, so $B_f=0$.

For the parameter family use \eqref{eq:5.12}. The coefficient of $|B_f(\nu)|^2$, after removing the positive factor $\beta v^{\beta-2}/s^2$, is
\begin{equation}\label{eq:rigidity-coefficient}
\frac{\widetilde\beta_f}{m-1}
=\frac{m}{m-1}+(\beta-2)\left(1-\frac{ks^2}{v^2}\right).
\end{equation}
This is increasing in $\beta$, and at the lower endpoint it equals
\[
\left.\frac{\widetilde\beta_f}{m-1}\right|_{\beta=(m-2)/(m-1)}
=\frac{m}{m-1}\frac{ks^2}{v^2}>0.
\]
Thus equality forces $B_f(\nu)=0$, the trace square to vanish, $\mathring\Pi=0$, and equality in the normal curvature term. The decomposition \eqref{eq:5.5} then gives $|B|^2+D_m=0$. Conversely, \eqref{eq:rigidity-squares} gives $B_f=0$ and hence $\nabla v=0$ by \eqref{eq:5.9}; \eqref{eq:5.13} now gives $\mathcal L_fv^\beta=0$. The positive $ks^2$ term is what prevents loss of the normal square at the endpoint.
\end{proof}

\begin{proposition}[local metric and weight]\label{Proposition 5.2}
Under the equality conditions of Proposition \ref{Proposition 5.1.}, there are local coordinates $(t,y)$, with $t$ increasing in the direction of $\nabla s$, a constant $c>0$, a metric $g_\Sigma$, and a function $f_\Sigma$ on the transverse level set such that
\begin{equation}\label{eq:rigidity-local-model}
\begin{gathered}
\bigl((s^2)'\bigr)^2=c s^2-ks^4,\qquad
s^2=\frac{c}{k}\sin^2\left(\frac{\sqrt{k}(t-t_0)}2\right),\\
g=dt^2+\bigl((s^2)'\bigr)^2g_\Sigma,\qquad
f=-(m-n)\log|(s^2)'|+f_\Sigma,\qquad v^2=c.
\end{gathered}
\end{equation}
The formula is restricted to an interval on which $(s^2)'>0$. A constant scaling of $g_\Sigma$ absorbs the choice of reference level. These are necessary local equality conditions; the tangential curvature bound still has to be imposed on a proposed model.
\end{proposition}
\begin{proof}
The Green identity gives
\[
m\psi=\Delta_f s^2
=\frac{m}{2s^2}|\nabla s^2|^2-\frac{mk}{2}s^2,
\quad\text{hence}\quad
\psi=\frac{|\nabla s^2|^2}{2s^2}-\frac{k}{2}s^2.
\]
For every tangent vector $Y$ to a level set,
\[
Y(|\nabla s^2|^2)=2\operatorname{Hess}s^2(Y,\nabla s^2)=0.
\]
Thus $|\nabla s^2|$ is locally a function of $s^2$. The normal field $\nu$ has geodesic integral curves, since the tangential part of $\nabla_\nu\nabla s^2=\psi\nu$ is zero. Choose arclength $t$ on these curves. Then $(s^2)'=|\nabla s^2|$ and
\[
(s^2)''=\psi=\frac{((s^2)')^2}{2s^2}-\frac{k}{2}s^2.
\]
Differentiate the quotient, using $(s^2)'>0$:
\[
\frac{d}{dt}\left(\frac{((s^2)')^2}{s^2}\right)
=\frac{2(s^2)'(s^2)''}{s^2}-\frac{((s^2)')^3}{s^4}
=-k(s^2)'.
\]
Integration yields $((s^2)')^2/s^2+ks^2=c$. The left side is positive, so $c>0$. Substitution verifies the sine-square solution in \eqref{eq:rigidity-local-model}; the general increasing solution differs only by a translation in $t$.

Writing $g=dt^2+g_t$, for tangent vectors $Y,Z$ the Hessian equation gives
\[
\Pi(Y,Z)=\frac{\psi}{(s^2)'}g_t(Y,Z),\qquad
\partial_tg_t=2\frac{(s^2)''}{(s^2)'}g_t.
\]
Integrating this scalar metric equation gives $g_t=((s^2)')^2g_\Sigma$. The weight equation is
\[
(s^2)'\partial_tf=-(m-n)(s^2)'',
\]
which integrates to the stated expression for $f$. Finally,
\[
v^2=4|\nabla s|^2+ks^2
=\frac{|\nabla s^2|^2}{s^2}+ks^2=c,
\]
and the weighted volume element is
\begin{equation}\label{eq:rigidity-density}
e^{-f}dV=|(s^2)'|^{m-1}\,dt\,e^{-f_\Sigma}dV_\Sigma.
\end{equation}
This completes the derivation of both the metric and the weight.
\end{proof}

\begin{theorem}[strictness in the smooth closed problem]\label{Theorem 5.3}
For every admissible regular radius, $A_f'<0$. Whenever $V_f$ is defined as in Theorem \ref{Theorem 3.11.},
\[
V_f'<0,\qquad (A_f-2(m-1)V_f)'>0.
\]
In the ranges of Theorem \ref{Theorem 4.8.} and Corollary \ref{Corollary 4.12.}, respectively, the same assertions hold for $A_{f,\beta}$ and for the pair $V_{f,\beta}$, $A_{f,\beta}-2(m-2)V_{f,\beta}$. The locally absolutely continuous representatives are strictly monotone on their admissible intervals.
\end{theorem}
\begin{proof}
Every nonempty exterior region $\{s>\operatorname{sn}_k(r/2)\}$ contains a point $x$ at which $s$ attains its global maximum. It is away from $p$. There $\nabla s=0$, and
\[
\Delta s^2=\Delta_fs^2=-\frac{mk}{2}s^2\ne0,\qquad
D_m(x)=\frac{m-n}{mn}(\Delta s^2)^2>0.
\]
Consequently the bracket of \eqref{eq:3.16} is positive on a neighborhood of $x$. Its exterior integral gives $A_f'<0$. Also $\nabla v(x)=0$, and the normal-free identity \eqref{eq:5.13} gives
\[
\mathcal L_fv^\beta(x)
=\frac{\beta v^{\beta-2}}{s^2}\bigl(|B|^2+D_m\bigr)(x)>0.
\]
Together with nonnegativity everywhere, this proves $A_{f,\beta}'<0$ from \eqref{eq:5.18}.

Every positive sublevel contains a punctured neighborhood of $p$. If its basic defect integral vanished, Proposition \ref{Proposition 5.1.} would imply $\nabla v=0$ there. This contradicts $v\asymp\rho^{\alpha-1}\to\infty$. For the parameter defect the same conclusion follows from the strictly positive coefficient \eqref{eq:rigidity-coefficient}, as in Proposition \ref{Proposition 4.11.}  Therefore both difference derivatives are positive. Finally,
\[
2(m-1)V_f'<A_f'<0,\qquad
2(m-2)V_{f,\beta}'<A_{f,\beta}'<0
\]
whenever the respective quantities are defined. Integrating these almost-everywhere inequalities proves strict monotonicity.
\end{proof}

If a local equality region were continued to a pole with $s=0$, \eqref{eq:rigidity-local-model} would give $s^2\sim c(t-t_0)^2/4$ and $f=-(m-n)\log|t-t_0|+O(1)$. This weight is singular for $m>n$, and its measure has radial exponent $m-1$ instead of the smooth exponent $n-1$. Hence the obstruction is not the mere use of the non-trace-free tensor $B_f$: the full square decomposition recovers equality information, but its geometry is incompatible with the prescribed smooth weight and point source.

The unweighted case $m=n$ with constant $f$ is a separate case. Starting with the ordinary Bochner formula removes $D_m$ and avoids division by $m-n$. A smoothly completed equality metric with a pole then has round transverse metric, and the sine warping yields sectional curvature $k$, as in the spherical rigidity underlying Manea's formulas. This statement requires that smooth completion; a local warped expression by itself is not a global sphere theorem.

\section{Negative effective dimension}\label{sec:negative}

Assume in this section that $m<0$, $n\ge3$, $k>0$, and
$\operatorname{Ric}_f^m\ge(m-1)kg$ on a closed weighted Riemannian manifold. The operator and source normalization are still \eqref{eq:2.4}--\eqref{eq:2.5}. Its potential is positive because $m(m-2)>0$, so its positive Green function exists. The definitions of $s,v,D_m,\mathcal L_f,B,B_f$ are unchanged. Negative effective dimensions also occur in the curvature-dimension theory studied by Ohta \cite{Ohta}; our argument below is a direct smooth calculation.

\begin{proposition}[Signs in the local identities]
The identities of Sections~\ref{sec:setup} and \ref{sec:beta} involving only local differentiation remain valid for $m<0$. In particular $\mathcal L_fv^\beta\ge0$ if $\beta\ge(m-2)/(m-1)$.
\end{proposition}
\begin{proof}
The chain rule and square completion are algebraic identities with nonzero denominators. The coefficient of $D_m$ is positive since $m(m-n)>0$. Thus \eqref{eq:3.6} and \eqref{eq:3.16} have the same nonnegative defect. The alternative identity
\[
|B|^2+D_m=|B_f|^2+\frac{(\operatorname{tr}_gB_f)^2}{m-n}
\]
is still true, but its last coefficient is now negative; it is not a nonnegative decomposition. Instead use the exact decomposition \eqref{eq:5.5}. Its coefficients satisfy
\[
\frac{m}{m-1}>0,\qquad
\frac{m-2}{m-1}>0,\qquad
\frac{m-1}{(n-1)(m-n)}>0.
\]
For the parameter formula the relevant coefficient is $\widetilde\beta_f/(m-1)$, not $\widetilde\beta_f$ alone. By \eqref{eq:rigidity-coefficient}, for $\beta\ge(m-2)/(m-1)$,
\[
\frac{\widetilde\beta_f}{m-1}
\ge\frac{m}{m-1}\frac{ks^2}{v^2}>0.
\]
The other coefficients in \eqref{eq:5.12} have the signs just listed. The curvature excess is nonnegative by assumption. Therefore $\mathcal L_fv^\beta\ge0$ on the regular set and, by the continuous identity \eqref{eq:5.13}, everywhere away from $p$. The equality conditions on regular regions are the same as in Proposition \ref{Proposition 5.1.}.
\end{proof}

\paragraph{The pole and the natural level parameter.}
The fundamental solution still satisfies $G\sim c_p\rho^{2-n}$. Now
\begin{equation}\label{eq:negative-pole}
2s=G^{1/(2-m)}\sim c_p^{1/(2-m)}\rho^{(2-n)/(2-m)}\longrightarrow\infty.
\end{equation}
Thus $s$ has a positive minimum $s_0$ away from $p$, and each finite sublevel is compactly contained in $M\setminus\{p\}$. We use $a$, the level value of $2s$ and write $a_0=2s_0$. The inverse sine need not exist on these levels, so no global $r$ parametrization is assumed. The symbol $a$ denotes a real parameter, not a replacement name for the function $s$.

For $a>a_0$ define the same integral expressions in this parameter:
\begin{equation}\label{eq:negative-definitions}
\begin{aligned}
I_{f,u}(a)&=a^{1-m}\int_{2s=a}2u|\nabla s|e^{-f}d\sigma
 +\frac{mk}{4}\int_{2s<a}uG e^{-f}dV,\\
J_{f,u}(a)&=a^{-m}\int_{2s<a}u(4|\nabla s|^2-ks^2)e^{-f}dV
 +\frac{k}{4}\int_{2s<a}uG e^{-f}dV,\\
W_{f,u}(a)&=a^{2-m}\int_{2s<a}\frac{u}{(2s)^2}
 \left(4|\nabla s|^2-\frac{mk}{m-2}s^2\right)e^{-f}dV+\frac{mk}{4(m-2)}\int_{2s<a}uG e^{-f}dV.
\end{aligned}
\end{equation}
All these integrals are finite for smooth $u$ on the punctured manifold. Put
\[
A_f=I_{f,v^2},\quad V_f=J_{f,v^2},\qquad
A_{f,\beta}=I_{f,v^\beta},\quad V_{f,\beta}=W_{f,v^\beta}.
\]
The primes in this section mean differentiation with respect to $a$.

\begin{proposition}[Flux and volume derivatives]\label{Proposition 6.2}
At regular values,
\begin{equation}\label{eq:negative-flux}
\begin{aligned}
I_{f,u}'=&a^{1-m}\int_{2s=a}\langle\nabla u,\nu\rangle e^{-f}d\sigma,\\
aJ_{f,u}'=&I_{f,u}-mJ_{f,u},\\
aW_{f,u}'=&I_{f,u}-(m-2)W_{f,u},\\
a^{3-m}I_{f,u}'=&\int_{2s<a}G^2\mathcal L_fu\,e^{-f}dV.
\end{aligned}
\end{equation}
In addition,
\begin{equation}\label{eq:negative-drift}
\int_{2s<a}\mathcal L_fu\,e^{-f}dV
=a^{m-1}\bigl(I_{f,u}'-2(m-2)W_{f,u}'\bigr).
\end{equation}
\end{proposition}
\begin{proof}
Since $\nabla G=(2-m)(2s)^{1-m}\nabla(2s)$, the boundary term of $I_{f,u}$ is
\[
\frac1{2-m}\int_{2s=a}u\langle\nabla G,\nu\rangle e^{-f}d\sigma.
\]
Differentiation by divergence and coarea gives
\[
\begin{aligned}
\frac{d}{da}\left(\frac1{2-m}\int_{2s=a}u\langle\nabla G,\nu\rangle e^{-f}d\sigma\right)
&=\frac1{2-m}\int_{2s=a}\frac{u\Delta_fG+\langle\nabla u,\nabla G\rangle}{2|\nabla s|}e^{-f}d\sigma\\
&=-\frac{mk}{4}\int_{2s=a}\frac{uG}{2|\nabla s|}e^{-f}d\sigma
 +a^{1-m}\int_{2s=a}\langle\nabla u,\nu\rangle e^{-f}d\sigma.
\end{aligned}
\]
The first term cancels the derivative of the volume correction. The coarea computations for $J,W$ are exactly those of Proposition \ref{Proposition 3.8.} and Proposition \ref{Proposition 4.9.} with $da/dr$ removed: their potential coefficients cancel as $mk/4-mk/4=0$ and $mk/4-(m-2)mk/[4(m-2)]=0$, respectively. These yield the middle identities of \eqref{eq:negative-flux}.

The divergence form of $\mathcal L_f$ gives, with the outward normal $\nu=\nabla s/|\nabla s|$,
\[
\begin{aligned}
\int_{2s<a}G^2\mathcal L_fu\,e^{-f}dV
=a^{4-2m}\int_{2s=a}\langle\nabla u,\nu\rangle e^{-f}d\sigma
=a^{3-m}I_{f,u}'.
\end{aligned}
\]
There is no inner pole boundary. Finally, \eqref{eq:5.27} and integration by parts give
\[
\begin{aligned}
\int_{2s<a}\left\langle\frac{\nabla s^2}{s^2},\nabla u\right\rangle e^{-f}dV
=2a^{m-2}\bigl(I_{f,u}-(m-2)W_{f,u}\bigr)
=2a^{m-1}W_{f,u}'.
\end{aligned}
\]
Here the $uG$ coefficients cancel exactly as above. Also $\int_{2s<a}\Delta_fu\,e^{-f}dV=a^{m-1}I_{f,u}'$. Substitution into the definition of $\mathcal L_f$ proves \eqref{eq:negative-drift}.
\end{proof}

\begin{theorem}[Three formulas for $m<0$]\label{Theorem 6.3}
For every $\beta\ge(m-2)/(m-1)$ and every regular $a>a_0$,
\begin{equation}\label{eq:negative-beta}
\begin{aligned}
A_{f,\beta}'(a)&=a^{m-3}\int_{2s<a}G^2\mathcal L_fv^\beta e^{-f}dV>0,\\
\bigl(A_{f,\beta}-2(m-2)V_{f,\beta}\bigr)'(a)
&=a^{1-m}\int_{2s<a}\mathcal L_fv^\beta e^{-f}dV>0,\\
V_{f,\beta}'(a)&>0.
\end{aligned}
\end{equation}
For the basic volume kernel the corresponding formulas are
\begin{equation}\label{eq:negative-basic}
\begin{aligned}
A_f'(a)&=8a^{m-3}\int_{2s<a}(2s)^{2-2m}
 \Bigl[|B|^2+D_m+(\operatorname{Ric}_f^m-(m-1)kg)(\nabla s^2,\nabla s^2)\Bigr]e^{-f}dV>0,\\
\bigl(A_f-2(m-1)V_f\bigr)'(a)
&=\frac8{a^{m+1}}\int_{2s<a}\Bigl[|B|^2+D_m+(\operatorname{Ric}_f^m-(m-1)kg)(\nabla s^2,\nabla s^2)\Bigr]e^{-f}dV>0,\\
V_f'(a)&>0.
\end{aligned}
\end{equation}
All these quantities are locally absolutely continuous, and the strict monotonicity extends across critical values. No pole upper bound on $\beta$ is needed for finite $a$.
\end{theorem}

\begin{proof}
The first two identities in \eqref{eq:negative-beta} follow from Proposition \ref{Proposition 6.2} with $u=v^\beta$. For the basic area formula substitute \eqref{eq:3.16} and $G^2/s^2=4(2s)^{2-2m}$. To check the second basic identity explicitly, integrate the divergence identity \eqref{eq:4.23} on $\{2s<a\}$. Its boundary terms are
\[
\frac{a^{m+1}}4A_f'
-\frac{m-1}{2}a^m\left(A_f-\frac{mk}{4}\int_{2s<a}v^2G e^{-f}dV\right).
\]
Its product-volume term is
\[
-\frac{m(m-1)}2a^m\left(V_f-\frac{k}{4}\int_{2s<a}v^2G e^{-f}dV\right).
\]
The two $v^2G$ coefficients cancel. Thus twice the defect integral is
\[
\frac{a^{m+1}}4A_f'-\frac{m-1}{2}a^m(A_f-mV_f)
=\frac{a^{m+1}}4\bigl(A_f'-2(m-1)V_f'\bigr),
\]
where $aV_f'=A_f-mV_f$ was used. This proves \eqref{eq:negative-basic} except for the volume sign.

The minimum set of $s$ has measure zero. Indeed $\nabla s=0$ there but $\Delta_fs=-mks/4>0$. The integrands are smooth on a fixed neighborhood of this set, so all volume integrals tend to zero as $a\downarrow a_0$. The boundary flux also tends to zero by writing it as the integral of $\operatorname{div}_f(u\nabla(2s))$. Therefore $A_f,A_{f,\beta},V_f,V_{f,\beta}$ all tend to zero at $a_0$.

The nonnegative derivative formulas imply $A_f,A_{f,\beta}\ge0$. Solving the volume equations with their zero initial values gives
\begin{equation}\label{eq:negative-averages}
\begin{aligned}
V_f(a)&=a^{-m}\int_{a_0}^a t^{m-1}A_f(t)\,dt\ge0,\\
V_{f,\beta}(a)&=a^{2-m}\int_{a_0}^a t^{m-3}A_{f,\beta}(t)\,dt\ge0.
\end{aligned}
\end{equation}
Since $m<0$ and $m-2<0$, the identities
$aV_f'=A_f-mV_f$ and $aV_{f,\beta}'=A_{f,\beta}-(m-2)V_{f,\beta}$ give the volume signs.

For strictness, every sublevel with $a>a_0$ contains a minimum point $x$. At that point
\[
\Delta s^2=-\frac{mk}{2}s^2>0,\quad
D_m(x)=\frac{m-n}{mn}(\Delta s^2)^2>0,\quad \nabla v(x)=0.
\]
Equation \eqref{eq:5.13} consequently gives $\mathcal L_fv^\beta(x)>0$. Continuity yields strict positivity of all defect integrals, followed by $A_f,A_{f,\beta}>0$ and strict positivity of the volume derivatives. Local absolute continuity follows from the same coarea and divergence argument as before, on compact sets away from $p$.
\end{proof}

The change of direction in the area formula comes from the integration domain and its orientation: for $m>n$ one integrates over the exterior to avoid the pole, whose inner normal is $-\nu$; for $m<0$ the finite sublevel itself avoids the pole and its outward normal is $\nu$. The local Bochner sign has not changed. Statements about the limit $a\to\infty$ would again require estimates at the pole and are not part of Theorem \ref{Theorem 6.3}.

\section{Comparison with earlier work and further directions}\label{sec:comparison}

\paragraph{Unweighted and weighted precedents.}
The calculations retain the Green-function, Bochner, and level-set route of Colding, Colding--Minicozzi, and Manea \cite{Colding,ColdingMinicozzi,Manea}. Their role here is structural: the potential in $L_f$, the spherical variable, and the curvature corrections are adapted to the positive lower bound. Our additional terms record the difference between the metric trace and the diffusion trace. They must be kept both in the sign argument and in equality analysis. The case $m=n$, $f$ constant is recovered by restarting the ordinary Bochner calculation, rather than substituting into fractions containing $m-n$.

Song--Wei--Wu \cite{SongWeiWu} use $N$ for the dimension excess, so their total effective dimension $n+N$ corresponds to our $m$. Their formulas concern the weighted Laplacian on nonparabolic spaces; they introduce several independent exponents and explicitly impose a pole-integrability condition. Their introductory finite-dimensional monotonicity theorem uses $\beta\ge2$, and their later discussion also addresses smaller exponents. Thus neither weighted monotonicity nor the Green-pole obstruction is new here.

For comparison of the convergence conditions, their condition is
\[
(n-2)(q-p)-\beta(q-n)>0,
\]
where $q$ here denotes their Green-power parameter (called $k$ in their paper). Taking $q=m$, the choices $p=0$ and $p=2$ become, respectively,
\[
n+2\alpha-2-\beta(1-\alpha)>0,
\qquad n-2-\beta(1-\alpha)>0.
\]
For $\beta=2$, the first is exactly $n+4\alpha-4>0$, the condition for our basic volume kernel. The second is the condition for the leading term of $W_{f,v^\beta}$. In the present formulas the additional $v^\beta G$ correction also requires $\beta(1-\alpha)<2$. These comparisons follow directly from the powers in the radial integrals; they explain precisely which restrictions are shared and which arise from the positive-potential correction.

Methodologically, our proof keeps the exact trace-free Hessian and the extra square $D_m$, then performs the full level-set decomposition. This yields the effective-dimensional sufficient threshold $\beta\ge(m-2)/(m-1)$ for the present positive-curvature density, including its equality conditions. We do not claim that the exponent alone is a new phenomenon: refined Hessian estimates already underlie the parameter families in the preceding literature. Our results concern different operators and integral quantities, and do not contain all of the independent parameter choices or infinite-dimensional statements of Song--Wei--Wu.

In an unpublished paper \cite{WYZ}, the first author also obtained three monotonicity formulas for the weighted Green function on weighted Riemannian manifolds with nonnegative $m$-dimensional Bakry–Émery Ricci curvature, which can be regarded as a weighted analogue of the work of Colding–Minicozzi \cite{ColdingMinicozzi}.

\paragraph{Entropy monotonicity formulas}
A complementary source of monotonicity formulas in geometric analysis is the entropy method for parabolic equations. Perelman \cite{Perelman} introduced the \(\mathcal W\)-entropy for the Ricci flow coupled with the conjugate heat equation and established its monotonicity without a curvature sign assumption. This formula played a fundamental role in his analysis of noncollapsing and singularity formation. For a fixed Riemannian metric, L. Ni \cite{NL, NL2} developed an analogous entropy formula for the linear heat equation, obtaining monotonicity under nonnegative Ricci curvature. These results exhibit a common mechanism: the derivative of a suitably normalized functional is expressed through geometric quantities whose signs are controlled by the evolution equation or by a curvature assumption.
The corresponding theory on weighted Riemannian manifolds was developed by X.-D. Li \cite{LiXD}, who established a \(\mathcal W\)-entropy formula for the heat equation associated with the Witten Laplacian under suitable geometric hypotheses. In this setting, the finite-dimensional Bakry–Émery tensor replaces the Ricci tensor, and an additional square term records the discrepancy between the actual dimension and the effective dimension. 

%Further developments include Li’s work on Hamilton’s Harnack inequality and entropy formulas on complete manifolds [5], and the work of Songzi Li and Xiang-Dong Li on time-dependent metrics and potentials, super Ricci flows, and curvature-dimension conditions allowing negative curvature lower bounds [6,7]. Their warped-product interpretation also clarifies the geometric origin of the additional terms in the weighted entropy formula.

The elliptic monotonicity formulas of Colding and Colding–Minicozzi provide a parallel approach in which the evolving time parameter is replaced by the level-set parameter of a Green function or a positive harmonic function. Their structural relationship with parabolic entropy formulas lies in the use of Bochner identities, curvature terms, and Hessian defects to express monotonicity and detect geometric equality. 

This analogy motivates the present investigation of weighted elliptic formulas under a positive Bakry–Émery curvature lower bound. Nevertheless, the elliptic problem requires a separate analysis: its derivatives are obtained through level-set fluxes and coarea identities, and the persistent Green-function singularity imposes integrability and inner-boundary conditions that are not resolved by the parabolic entropy formulas. In particular, the effective dimension entering the curvature decomposition need not agree with the actual dimension governing the Green pole. The present work therefore develops the elliptic counterpart of this weighted Bochner mechanism while explicitly tracking the pole contributions and the resulting restrictions on monotonicity and equality.

Entropy monotonicity also extends to curvature lower bounds that allow negative curvature. J. Li and X. Xu \cite{LiXu} established differential Harnack inequalities and curvature-dependent entropy formulas for the linear heat equation under \(\operatorname{Ric}\ge-Kg\). Their Perelman-type entropy incorporates explicit time-dependent corrections determined by \(K\), and its derivative is expressed through a shifted Hessian square and the nonnegative curvature term \(\operatorname{Ric}+Kg\). At \(K=0\), this functional reduces to Ni’s entropy. Thus their work shows how the normalization of the entropy can compensate for a negative curvature lower bound.
S. Li and X.-D. Li \cite{LL} developed the corresponding weighted formula under \(\operatorname{Ric}_f^m\ge-Kg\), with \(m\ge n\) and \(K\ge0\), subject to the stated regularity and geometric assumptions. Their entropy dissipation formula contains the curvature excess \(\operatorname{Ric}_f^m+Kg\), a shifted Hessian square, and an additional square involving the weight and the dimension difference \(m-n\). They also extended this construction to suitable time-dependent metrics and potentials. These formulas provide a parabolic precedent for adapting monotonic quantities to a prescribed curvature lower bound. Here a negative curvature lower bound should be distinguished from a negative effective dimension: the cited results retain \(m\ge n\), whereas the case \(m<0\) considered in the present paper changes the dimension parameter itself.

The relationship with a positive curvature lower bound requires some care. A manifold satisfying \(\operatorname{Ric}_f^m\ge(m-1)kg\), with \(m>n\) and \(k>0\), also satisfies the nonnegative curvature condition, so the usual weighted entropy monotonicity remains available by taking \(K=0\). This application, however, does not exploit the magnitude of the positive lower bound or provide a spherical normalization. In the present elliptic problem, the positive constant enters both the Green operator \(-\Delta_f+m(m-2)k/4\) and the spherical change of variables, producing curvature-correcting level-set quantities. Its role therefore differs from the time-dependent correction used to compensate for a negative lower bound in the heat equation.

The common principle is the completion of geometric squares relative to the curvature bound, but the analytic constructions and equality conditions must be examined separately. In particular, the negative-lower-bound entropy formulas do not automatically identify hyperbolic space as an equality model, just as applying the nonnegative-curvature entropy formula on a positively curved manifold does not yield spherical rigidity. For the weighted elliptic formulas studied here, the compatibility of the equality equations with the smooth weight and the actual-dimensional Green singularity is an additional essential issue.

\paragraph{Positive-curvature $\mathrm{RCD}$ spaces.}
Gigli--Violo \cite{GigliViolo} established harmonic-function monotonicity in $\mathrm{RCD}(0,N)$ spaces, including rigidity and almost-rigidity questions. A natural next problem is a positive-curvature analogue on $\mathrm{RCD}((m-1)k,m)$ spaces with $m>2$. The natural candidate is a resolvent Green kernel for $-\Delta+m(m-2)k/4$, with $s=\tfrac12G^{1/(2-m)}$ and the same corrected integral kernels. The smooth identities suggest a proof strategy, but do not by themselves establish this extension.

At least four points require work. First, one must justify the kernel's pole behavior and the integrability of the corrected quantities; the local dimension need not equal the synthetic upper dimension. Second, the Bochner computation must be formulated using a measure-valued $\Gamma_2$ and the available Hessian calculus; the smooth term $df\otimes df/(m-n)$ cannot simply be written on an arbitrary $\mathrm{RCD}$ space. Third, level-set fluxes should be replaced by weak divergence and perimeter/coarea identities with the necessary representatives. Fourth, equality would have to be converted into a metric-measure suspension or warped-product conclusion, with the singular tips and their measures treated explicitly. These are proposed research problems, not conclusions of this paper. The $m<0$ calculation should likewise not be interpreted as a theorem in the usual positive-dimensional $\mathrm{RCD}$ class.

	\medskip
	\noindent\textbf{Disclosure on AI assistance.}

The authors used ChatGPT-6 Astra to assist with the calculation and derivation of some results. All proof ideas, key innovations, and scholarly contributions are the authors’ own. All content was reviewed, verified, and revised by the authors, who take full responsibility for its accuracy, originality, and entire content.

Yu-Zhao Wang, School of Mathematics and Statistics, Shanxi University, Taiyuan, 030006, Shanxi, China\\
Email: wangyuzhao@sxu.edu.cn\\

Zeng-Ting Wu, School of Mathematics and Statistics, Shanxi University, Taiyuan, 030006, Shanxi, China\\
Email: wuzengting@sxu.edu.cn
\end{document}